\documentclass[epsfig,11pt]{llncs}
\usepackage{amsfonts}

\newcommand{\bbV}{{\mathcal V}}
\newcommand{\bbI}{{\mathcal I}}

\newcommand{\bbH}{{\mathcal H}}
\newcommand{\remove}[1]{}
\usepackage{epsfig,latexsym,graphicx}
\newtheorem{observation}{Observation}

\newtheorem{assumption}{Assumption}

\begin{document}
\title{Disjoint Empty Convex Pentagons in Planar Point Sets}

\author{Bhaswar B. Bhattacharya$^1$\and Sandip Das$^2$}
\institute
{$^1$ \small {Indian Statistical Institute, Kolkata, India, {\tt bhaswar.bhattacharya@gmail.com}}\\
$^2$ \small{Advanced Computing and Microelectronics Unit, Indian Statistical Institute, Kolkata, India,\\
{\tt sandipdas@isical.ac.in}}}

\maketitle

\begin{abstract}
Harborth [{\it Elemente der Mathematik}, Vol. 33 (5), 116--118, 1978] proved
that every set of 10 points in the plane, no three on a line, contains an empty
convex pentagon. From this it follows that the number of disjoint empty convex
pentagons in any set of $n$ points in the plane is least
$\lfloor\frac{n}{10}\rfloor$. In this paper we prove that every set of 19
points in the plane, no three on a line, contains two disjoint empty convex
pentagons. We also show that any set of $2m+9$ points in the plane, where $m$
is a positive integer, can be subdivided into three disjoint convex regions,
two of which contains $m$ points each, and another contains a set of 9 points
containing an empty convex pentagon. Combining these two results, we obtain
non-trivial lower bounds on the number of disjoint empty convex pentagons in
planar points sets. We show that the number of disjoint empty convex pentagons
in any set of $n$ points in the plane, no three on a line, is at least
$\lfloor\frac{5n}{47}\rfloor$. This bound has been further improved to
$\frac{3n-1}{28}$ for infinitely many $n$. 
\\

\end{abstract}

\section{Introduction}
\label{sec:intro}

The origin of the problems concerning the existence of empty convex polygons
goes back to the famous theorem due to Erd\H os and Szekeres \cite{erdos}. It
states that for every positive integer $m\geq 3$, there exits a smallest
integer $ES(m)$, such that any set of $n$ points $ (n \geq ES(m))$ in the
plane, no three on a line, contains a subset of $m$ points which lie on the
vertices of a convex polygon.  Evaluating the exact value of $ES(m)$ is a long
standing open problem. A construction due to Erd\H os \cite{erdosz} shows that
$ES(m)\geq 2^{m-2} +1$, which is also conjectured to be sharp. It is known that
$ES(4)=5$ and $ES(5)=9$ \cite{kalb}. Following a long computer search, Szekeres
and Peters \cite{szekeres} recently proved that $ES(6)=17$. The value of
$ES(m)$ is unknown for all $m> 6$. The best known upper bound for $m\geq 7$ is
due to T\'oth and Valtr \cite{tothvaltr} - $ES(m) \leq {{2m-5}\choose {m-3}}+1$.
For a more detailed description of the Erd\H os-Szekeres theorem and its
numerous ramifications see the surveys by B\'ar\'any and K\'arolyi
\cite{problemsandresults} and Morris and Soltan \cite{survey}.

In 1978, Erd\H os \cite{erdosempty} asked whether for every positive integer
$k$, there exists a smallest integer $H(k)$, such that any set of at least
$H(k)$ points in the plane, no three on a line, contains $k$ points which lie
on the vertices of a convex polygon whose interior contains no points of the set.
Such a subset is called an {\it empty convex $k$-gon} or a {\it k-hole}. Esther
Klein showed {$ H(4)=5$} and Harborth \cite{harborth} proved that {$ H(5)=10$}.
Horton \cite{horton} showed that it is possible to construct arbitrarily large
set of points without a 7-hole, thereby proving that {$ H(k) $} does not exist
for {$ k \geq 7$}. Recently, after a long wait, the existence of $ H(6)$ has
been proved by Gerken \cite{gerken} and independently by Nicol\'as
\cite{nicolas}. Later Valtr \cite{valtrhexagon} gave a simpler version of
Gerken's proof. For results regarding the number of $k$-holes in planar point
sets and other related problems see \cite{baranycanadianbulletin,problemsandresults,baranyvaltr,dumitrescu,sharir}.
Existence of a hole of any fixed size in sufficiently large point sets, with some additional restrictions on the point sets, has been studied by
K\'arolyi et al. \cite{karolyi_almost_empty_exact,karolyi_modular}, Kun and Lippner \cite{kun_lippner}, and Valtr \cite{valtr_sufficient}.

Two empty convex polygons are said to be {\it disjoint} if their convex hulls
do not intersect. For positive integers $k\leq \ell$, denote by $H(k,\ell)$ the smallest integer
such that any set of $H(k, \ell)$ points in the plane, no three on a line,
contains both a $k$-hole and a $\ell$-hole which are disjoint. Clearly,
$H(3,3)=6$ and Horton's result \cite{horton} implies that $H(k, \ell)$ does not
exist for all $\ell\geq 7$. Urabe \cite{urabedam} showed that $H(3, 4)=7$,
while Hosono and Urabe \cite{urabecgta} showed that $H(4,4)=9$. Hosono and
Urabe \cite{hosono} also proved that $H(3,5)=10$, $12\leq H(4,5)\leq 14$, and
$16\leq H(5,5)\leq 20$. The results $H(3, 4)=7$ and $H(4, 5)\leq 14$ were later
reconfirmed by Wu and Ding \cite{reconfirmation}. Using the
computer-aided order-type enumeration method, Aichholzer et al. \cite{toth}
proved that every set of 11 points in the plane, no three on a line, contains
either a 6-hole or a 5-hole and a disjoint 4-hole. Recently, this result was
proved geometrically by Bhattacharya and Das
\cite{bbbsd11point1,bbbsd11point2}. Using this Ramsey-type result, Hosono and
Urabe \cite{kyotocggt} proved that $H(4, 5)\leq 13$, which was later
tightened to $H(4, 5)=12$ by Bhattacharya and Das \cite{bbbsdqp}.
Hosono and Urabe \cite{kyotocggt} have also improved the lower bound on $H(5,
5)$ to 17.

The problems concerning disjoint holes was, in fact, first studied by Urabe
\cite{urabedam} while addressing the problem of partitioning of planar point
sets. For any set $S$ of points in the plane, denote by $CH(S)$ the {\it convex
hull} of $S$. Given a set $S$ of $n$ points in the plane, no three on a line, a
{\it disjoint convex partition} of $S$ is a partition of $S$ into subsets $S_1,
S_2, \ldots S_t$, with $\sum_{i=1}^t |S_i| = n$, such that for each $i\in\{1,
2, \ldots, t\}$, $CH(S_i)$ forms a $|S_i|$-gon and $CH(S_i)\cap
CH(S_j)=\emptyset$, for any pair of indices $i, j$. Observe that in any
disjoint convex partition of $S$, the set $S_i$ forms a $|S_i|$-hole and the
holes formed by the sets $S_i$ and $S_j$ are disjoint for any pair of distinct
indices $i, j$.
If $F(S)$ denote the minimum number of disjoint holes in any disjoint convex
partition of $S$, then $F(n)= \max_S F(S)$, where the maximum is taken over all
sets $S$ of $n$ points, is called the {\it disjoint convex partition number}
for all sets of fixed size $n$. The disjoint convex partition number $F(n)$ is
bounded by $\lceil\frac{n-1}{4}\rceil\leq F(n)\leq\lceil{\frac{5n}{18}}\rceil$.
The lower bound is by Urabe \cite{urabedam} and the upper bound by Hosono and
Urabe \cite{urabecgta}. The proof of the upper bound uses the fact that every
set of 7 points in the plane contains a 3-hole and a disjoint 4-hole. Later, Xu
and Ding \cite{empty_c_annals} improved the lower bound to
$\lceil\frac{n+1}{4}\rceil$. Recently, Aichholzer et al. \cite{toth} introduced
the notion pseudo-convex partitioning of planar point sets, which extends the
concept partitioning, in the sense, that they allow both convex polygons and
pseudo-triangles in the partition.

Urabe \cite{urabecgta} also defined the function $F_k(n)$ as the minimum number
of pairwise disjoint $k$-holes in any $n$-element point set. If $F_k(S)$ denotes the number of $k$-holes in a disjoint partition of $S$, then $F_k(n) =\min_S\{\max_{\pi_d}F_k(S)\}\}$,
where the maximum is taken over all disjoint partitions $\pi_d$ of $S$, and the minimum is
taken over all sets $S$ with $|S|=n$. Hosono and Urabe \cite{urabecgta}
proved any set of 9 points, no three on a line, contains two disjoint 4-holes.
They also showed any set of $2m+4$ points can be divided into three disjoint
convex regions, one containing a 4-hole and the others containing $m$ points each.
Combining these two results they proved
$F_4(n)\geq\lfloor\frac{5n}{22}\rfloor$. This bound can be improved to
$(3n-1)/13$ for infinitely many $n$.

The problem, however, appears to be much more complicated in the case of
disjoint 5-holes. Harborth's result \cite{harborth} implies
$F_5(n)\geq\lfloor\frac{n}{10}\rfloor$, which, to the best our knowledge, is
the only known lower bound on this number. A construction by Hosono and Urabe
\cite{kyotocggt} shows that $F_5(n) \leq 1$ if $n \leq 16$. In general, it is
known that $F_5(n) < n/6$ \cite{problemsandresults}. Moreover, Hosono and Urabe
\cite{urabecgta} states the impossibility of an analogous result for 5-holes
with $2m + 5$ points.

In this paper, following a couple of new results for small point sets, we prove
non-trivial lower bounds on $F_5(n)$. At first, we show that every set of 19
points in the plane, no three on a line, contains two disjoint 5-holes. In
other words, this implies, $F_5(19)\geq 2$ or $H(5, 5)\leq 19$. Drawing
parallel from the result of Hosono and Urabe \cite{urabecgta}, we also show
that any set of $2m+9$ points in the plane, where $m$ is a positive integer,
can be subdivided into three disjoint convex regions, two of which contains $m$
points each, and the third one is a set of 9 points containing a 5-hole. Combining these
two results, we prove $F_5(n)\geq\lfloor\frac{5n}{47}\rfloor$. This bound can
be further improved to $\frac{3n-1}{28}$ for infinitely many $n$. The proofs
rely on a series of results concerning the existence of 5-holes in planar point
sets having less than 10 points.

The paper is organized as follows. The results proving the existence of 5-holes
in point sets having less than 10 points, and the characterization of 9-point
sets not containing any 5-hole are presented in Section \ref{smallpointsets}.
In Section \ref{mainresults}, we give the formal statements of our main results
and use them to prove lower bounds on $F_5(n)$. The proofs of the 19-point
result and the $2m+9$-point partitioning theorem are presented in Sections
\ref{proofnineteen} and \ref{proof2m+9}, respectively. In Section \ref{dn} we
introduce notations and definitions and in Section \ref{c:conclusion} we
summarize our work and provide some directions for future work.

\section{Notations and Definitions}
\label{dn}

We first introduce the definitions and notations required for the remainder of
the paper. Let $S$ be a finite set of points in the plane in general position,
that is, no three on a line. Denote the convex hull of $S$ by $CH(S)$. The
boundary vertices of $CH(S)$, and the points of $S$ in the interior of $CH(S)$
are denoted by $\bbV(CH(S))$ and $\bbI (CH(S))$, respectively. A region $R$
in the plane is said to be empty in $S$, if $R$ contains no elements of $S$. A
point $p \in S$ is said to be {\it $k$-redundant} in a subset $T$ of $S$, if
there exists a $k$-hole in $T \backslash \{p\}$.

By $\mathcal P=p_1p_2\ldots p_k$ we denote a convex $k$-gon with vertices $p_1, p_2,
\ldots, p_k$ taken in the counter-clockwise order. $\bbV(\mathcal P)$ denotes the set
of vertices of $\mathcal P$ and $\bbI(\mathcal P)$ the interior of $\mathcal P$.


The {\it $j$-th convex layer} of $S$, denoted by $L\{j, S\}$, is the set of
points  that lie on the boundary of $CH(S\backslash\{\bigcup_{i=1}^{j-1}L\{i,
S\}\})$, where $L\{1, S\}=\bbV(CH(S))$.
If $p,~q\in S$ are such that ${pq}$ is an edge of the convex hull of the $j$-th
layer, then the open halfplane bounded by the line $pq$ and not containing any
point of $S\backslash\{\bigcup_{i=1}^{j-1}L\{i, S\}\}$ will be referred to as
the {\it outer} halfplane induced by the edge ${pq}$.

For any three points $p, q, r \in S$, $\bbH(pq, r)$ (respectively $\mathcal
H_c(pq, r)$) denotes the open (respectively closed) halfplane bounded by the
line $pq$ containing the point $r$. Similarly, $\overline{\bbH}(pq, r)$
(respectively $\overline \mathcal H_c(pq, r)$) is the open (respectively
closed) halfplane bounded by $pq$ not containing the point $r$.

Moreover, if $p, q, r\in S$ is such that $\angle rpq < \pi$, then $Cone(rpq)$
is the set of points in $\mathbb R^2$ which lies in the interior of the angular
domain $\angle rpq$. A point $s \in Cone(rpq)\cap S$ is called the {\it nearest
angular neighbor} of $\overrightarrow{pq}$ in $Cone(rpq)$ if $Cone(spq)$ is
empty in $S$. In general, whenever we have a convex region $R$, we think of $R$
as the set of points in $\mathbb R^2$ which lies in the region $R$. Thus, for
any convex region $R$ a point $s \in R\cap S$ is called the {\it nearest
angular neighbor} of $\overrightarrow{pq}$ in $R$ if $Cone(spq)\cap R$ is empty
in $S$. More generally, for any positive integer $k$, a point $s\in S$ is
called the {\it $k$-th angular neighbor} of $\overrightarrow{pq}$ whenever
$Cone(spq)\cap R$ contains exactly $k-1$ points of $S$ in its interior. Also,
for any convex region $R$, the point $s\in S$, which has the shortest
perpendicular distance to the line $pq$, $p, q\in S$, is called the {\it
nearest neighbor} of $pq$ in $R$.

\section{5-Holes With Less Than 10 Points}
\label{smallpointsets}

We begin by restating a well known result regarding the existence of 5-holes in
planar point sets.

\begin{lemma}\cite{matousek} Any set of points in general position containing a convex hexagon, contains a 5-hole.
\label{lm:lm1}
\end{lemma}

From the Erd\H os Szekeres theorem, we know that every sufficiently large set
of points in the plane in general position, contains a convex hexagon. Lemma
\ref{lm:lm1} therefore ensures that every sufficiently large set of points in
the plane contains a 5-hole. Harborth \cite{harborth} showed that a minimum of
10 points are required to ensure the existence of a 5-hole, that is $H(5)=10$.
This means, the existence of a 5-hole is not guaranteed if we have less than 10
points in the plane \cite{harborth}.

In the following, we prove two lemmas where we show, if the convex hull of the
point set is not a triangle, a 5-hole can be obtained in less than 10 points.

\begin{lemma}If $Z$ is a set of points in the plane in general position, with $|\bbV(CH(Z))|=5$ and $|\bbI(CH(Z))|\geq 2$, then
$Z$ contains a 5-hole. \label{lm:lm2}
\end{lemma}
\begin{proof}To begin with suppose there are only two points {$y_1$} and {$y_2$}
in $\bbI(CH(Z))$. The extended straight line {$y_1y_2$} divides the plane into
two halfplanes, one of which must contain at least three points of
{$\bbV(CH(Z))$}. These three points along with the points {$y_1$} and {$y_2$}
forms a 5-hole (Figure \ref{fig:fig12}(a)).

Next suppose, there are three points {$y_1$}, {$y_2$}, and {$y_3$} in
$\bbI(CH(Z))$. Consider the partition of the exterior of $ y_1y_2y_3$ into
disjoint regions $R_i$ as shown in Figure \ref{fig:fig12}(b). Let {$| R_{i} |$}
denote the number of points of {$\bbV(CH(Z))$} in region {$R_{i}$}. If $Z$ does
not contain a 5-hole, we must have:

\begin{figure*}[h]
\centering
\begin{minipage}[c]{0.33\textwidth}
\centering
\includegraphics[width=1.75in]
    {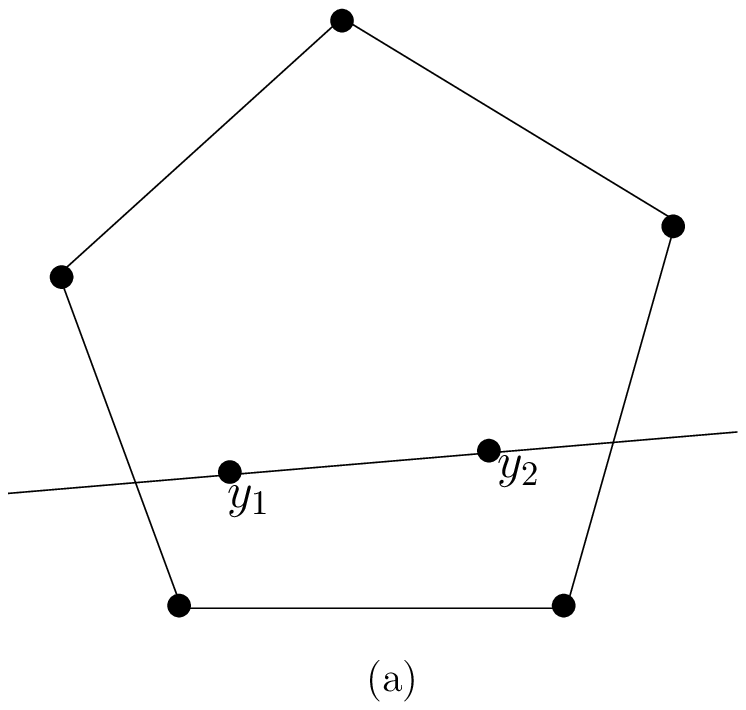}\\
\end{minipage}%
\begin{minipage}[c]{0.33\textwidth}
\centering
\includegraphics[width=1.75in]
    {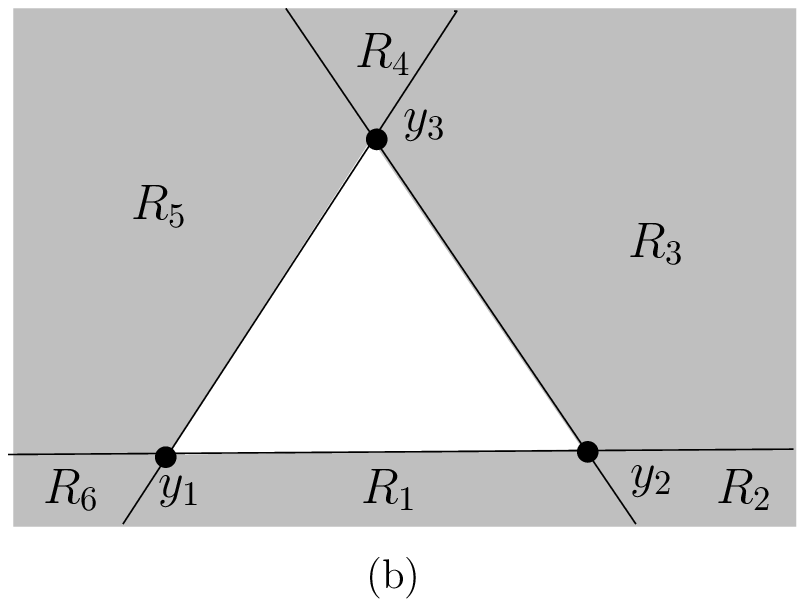}\\
\end{minipage}
\begin{minipage}[c]{0.33\textwidth}
\centering
\includegraphics[width=1.75in]
    {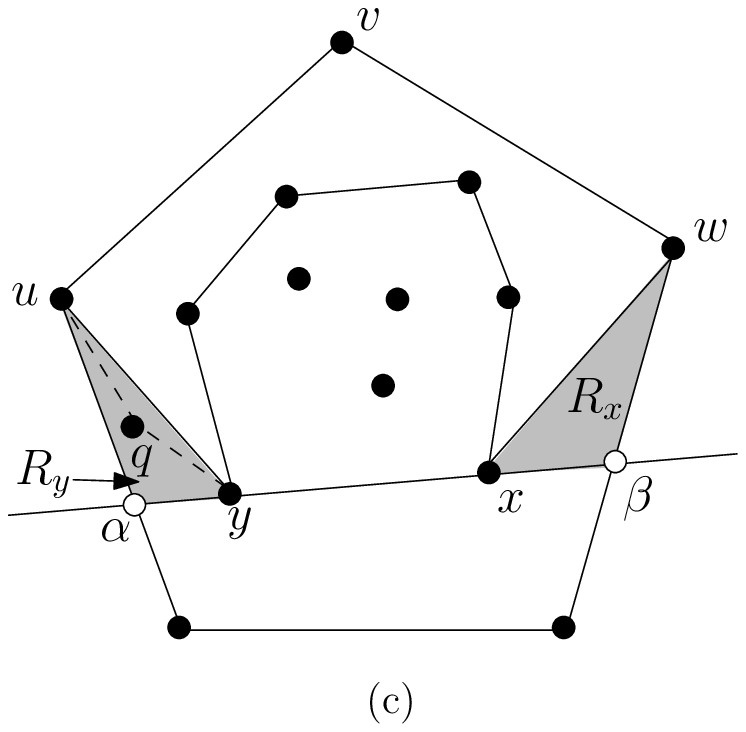}\\
\end{minipage}
\caption{Illustrations for the proof of Lemma \ref{lm:lm2}.}
\label{fig:fig12}\vspace{-0.15in}
\end{figure*}

\begin{equation}
| R_{1} | \leq 1, \hspace{1cm} | R_{3} | \leq 1, \hspace{1cm} | R_{5} | \\\leq
1, \label{eq:EN2}
\end{equation}
\vspace{-0.25in}
\begin{eqnarray}
| R_{6} | + | R_{1} | + | R_{2} | & \leq & 2, \nonumber \\
| R_{2} | + | R_{3} | + | R_{4} | & \leq & 2, \nonumber \\
| R_{4} | + | R_{5} | + | R_{6} | & \leq & 2. \label{eq:EN3}
\end{eqnarray}

Adding the inequalities of (\ref{eq:EN3}) and using the fact $|\bbV(CH(Z))|= 5$
we get {$ |R_{2}| + | R_{4}| + | R_{6}|  \leq 1$}. On adding this inequality
with those of (\ref{eq:EN2}) we finally get $\sum_{i=1}^6| R_{i} |\leq 4 <
5=|\bbV(CH(Z))|$, which is a contradiction.

Finally, suppose $|\bbI(CH(Z))|=k\geq 4$. Let $x, y\in Z$ be such that $xy$ is
an edge of $CH(\bbI (CH(Z)))$ and $z\in \bbI(CH(Z))$ be any other point.
If $|\bbV(CH(Z)) \cap \overline\mathcal H(xy, z)| \geq 3$, the points $x$ and $y$
together with the three points of $\bbV(CH(Z)) \cap \overline\mathcal H(xy, z)$
form a 5-hole. When $| \bbV(CH(Z)) \cap \overline\mathcal H(xy, z)| = 1$, the
4 points in  $\bbV(CH(Z)) \cap \mathcal H(xy, z)$ along with the points {\it x}
and {\it y} form a convex hexagon, which contains a 5-hole from Lemma
\ref{lm:lm1}. Otherwise, {$| \bbV(CH(Z)) \cap \overline\mathcal H(xy, z) | =
2$}. Denote by $\alpha, \beta$ the points where the extended straight line passing through the points $x$ and $y$ intersects the boundary of $CH(Z)$, as shown
in Figure \ref{fig:fig12}(c). Let $R_x=\bbI( wx\beta)$ and $R_y=\bbI( uy\alpha)$ be the two triangular
regions generated inside $CH(Z)$ in the halfplane $\mathcal H(xy, z)$.
If any one of $R_x$ or $R_y$ is non-empty in $Z$,
the nearest neighbor $q$ of the line $uy$ (or $wx$) in $R_y$ (or $R_x$) forms
the convex hexagon $uvwxyq$ (or $xyuvwq$), which contains an 5-hole from Lemma
\ref {lm:lm1}. Therefore, assume that both $R_x$ and $R_y$ are empty in $Z$.
Observe that the number of points of $Z$ inside $uvwxy$ is exactly two less
than the number of points of $Z$ inside $CH(Z)$. By applying this argument
repeatedly on the modified pentagon we finally get a 5-hole or a convex
pentagon with two or three interior points.\hfill $\Box$
\end{proof}

\begin{lemma}If $Z$ is a set of points in the plane in general position, with $|\bbV(CH(Z))|=4$ and $|\bbI(CH(Z))|\geq 5$,
then $Z$ contains a 5-hole. \label{lm:lm3}
\end{lemma}
\begin{proof}Let $CH(Z)$ be the polygon $p_1p_2p_3p_4$.
If some outer halfplane induced by an edge of $CH(\bbI(CH(Z)))$ contains more
than two points of $\bbV(CH(Z))$, then $Z$ contains a 5-hole. Therefore, we
assume

\begin{assumption}
Every outer halfplane induced by the edges of $CH(\bbI(CH(Z)))$ contains at
most two points of $\bbV(CH(Z))$. \label{assumption1}
\end{assumption}

To begin with suppose $|\bbI (CH(Z))|=5 $. If $|\bbV(CH(\bbI (CH(Z)))) | =5$, we are
done. Thus, the convex hull of the second layer of $Z$ is either a quadrilateral or
a triangle. Let $CH(\bbI(CH(Z)))$ be the polygon $z_1z_2\ldots z_k$, where $k$ is either 3 or 4.
This means $3\leq |L\{2, Z\}| \leq 4$, and we have the following two cases:

\begin{description}
\item[{\it Case} 1:] $|L\{2, Z\}|=4$. Let $x\in L\{3, Z\}$ and w. l. o. g. assume
$x\in\bbI( z_1z_3z_4)\cap Z$. Consider the partition of the exterior of the
quadrilateral $z_1z_2z_3z_4$ into disjoint regions $R_i$ as shown in Figure
\ref{fig:fig34}(a). Let $|R_i|$ denote the number of points of $\bbV(CH(Z))$ in
the region $R_i$. If there exists a point $p_i\in R_3\cap Z$, then
$p_iz_2z_1z_3x$ forms a 5-hole. Therefore, assume that $|R_3|=0$, and
similarly, $|R_5|=0$. Moreover, if $|R_1|+|R_2|\geq 2$, $((R_1\cup R_2)\cap
\bbV(CH(Z)))\cup\{z_1, z_4, x\}$ contains a 5-hole. This implies,
$|R_1|+|R_2|\leq 1$ and similarly $|R_6|+|R_7|\leq 1$. Therefore, $|R_4|\geq 2$
and Assumption \ref{assumption1} implies that $|R_4|=2$. This implies that the set of points in $(R_4\cap Z)\cup\{z_1, z_3, z_4\}$ forms a convex pentagon
with exactly two interior points, which then contains a 5-hole from Lemma
\ref{lm:lm2}.

\begin{figure*}[h]
\centering
\begin{minipage}[c]{0.33\textwidth}
\centering
\includegraphics[width=2.0in]
    {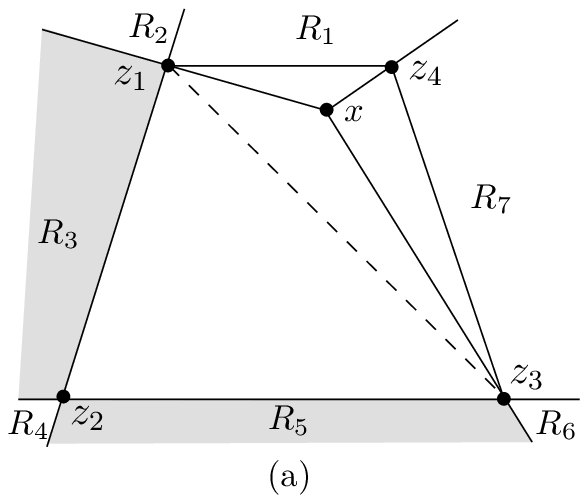}\\
\end{minipage}%
\begin{minipage}[c]{0.33\textwidth}
\centering
\includegraphics[width=2.0in]
    {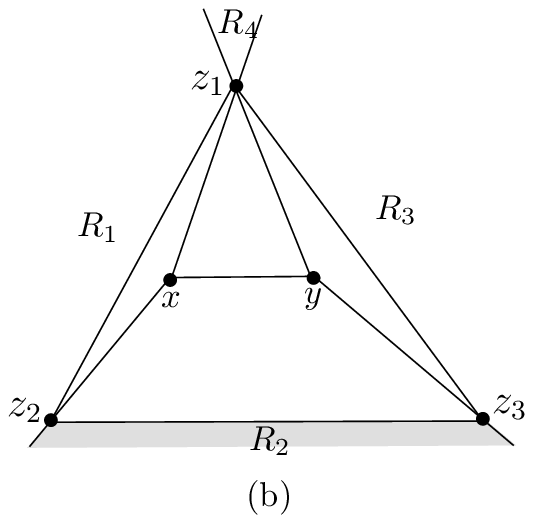}\\
\end{minipage}
\begin{minipage}[c]{0.33\textwidth}
\centering
\includegraphics[width=2.0in]
    {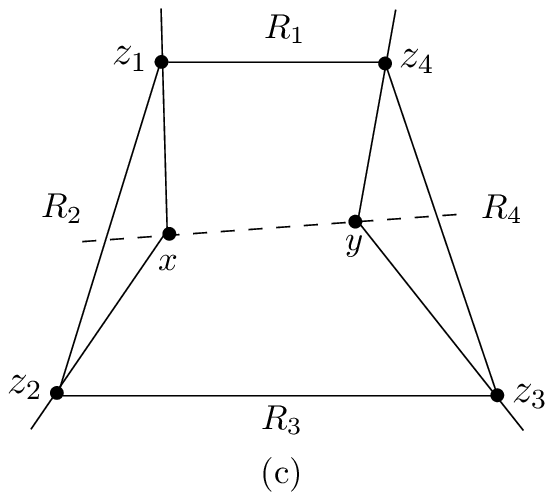}\\
\end{minipage}
\caption{Illustrations for the proof of Lemma \ref{lm:lm3}: (a) $|L\{2, Z\}|=4$, (b) $|L\{2, Z\}|=3$, (c) Illustration for the proof of Theorem \ref{th:layer}.}
  \label{fig:fig34}\vspace{-0.15in}
\end{figure*}

\item[{\it Case} 2:] $|L\{2, Z\}|=3$. Let $L\{3, Z\}=\{x, y\}$.
Consider the partition of the exterior of $CH(\bbI(CH(Z)))$ as shown in Figure
\ref{fig:fig34}(b). Observe that $Z$ contains a 5-hole unless $|R_2|=0$,
$|R_1|\leq 1$, and $|R_3|+|R_4|\leq 1$. This implies that $\sum_{i=1}^4
|R_i|\leq 3<4=|\bbV(CH(Z))|$, which is a contradiction.
\end{description}

Now, consider $|\bbI (CH(Z))|>5 $. W.l.o.g. assume that $\bbI(
p_1p_2p_3)\cap Z$ is non-empty. If $| CH(Z \backslash \{p_2\}) | \geq 5$, a
5-hole in $Z\backslash \{p_2\}$ is ensured from Lemma \ref{lm:lm1} and  Lemma
\ref{lm:lm2}. Otherwise, $CH(Z\backslash \{p_2\})$ is a quadrilateral with
exactly one less point of $Z$ in its interior than $CH(Z)$. By repeating this
process we finally get a convex quadrilateral with exactly 5 points in its
interior, thus reducing the problem to {\it Case} 1 and {\it Case} 2. \hfill
$\Box$
\end{proof}

From the argument at the end of the proof of the previous lemma, it follows
that if $|\bbI(CH(Z))|\geq 6$, then either $p_1$ or $p_3$ is 5-redundant in
$Z$. Similarly, either $p_2$ or $p_4$ is 5-redundant in $Z$. Therefore, we have
the following corollary:

\begin{corollary}
Let $Z$ be a set of points in the plane in general position, such that
$CH(Z)$ is the polygon $z_1z_2z_3z_4$, and $|\bbI(CH(Z))|\geq 6$. Then the following
statements hold:
\begin{description}
\item[{\it (i)}] If for some $i\in\{1, 2, 3, 4\}$, $\bbI(z_{i-1}z_iz_{i+1})\cap Z$ is non-empty, then $z_i$ is 5-redundant in $Z$, where the indices are taken modulo 4.
\item[{\it (ii)}]At least one of the vertices corresponding to any diagonal of $CH(Z)$ is
5-redundant in $Z$. \hfill$\Box$
\end{description}
\label{cor:5redun}
\end{corollary}

Moreover, by combining Lemmas \ref{lm:lm1}, \ref{lm:lm2}, and \ref{lm:lm3},
the following result about the existence of 5-holes is immediate.

\begin{corollary}Any set $Z$ of 9 points in the plane in general position, with $|\bbV(CH(Z))|\geq 4$, contains a 5-hole. \hfill $\Box$
\label{corollary:nine_points}
\end{corollary}

Two sets of points, $S_1$ and $S_2$, in general position, having the same
number of points belong to the same {\it layer equivalence class} if the number
of layers in both the point sets is the same and $|L\{k, S_1\}|=|L\{k, S_2\}|$,
for all $k$. A set $S$ of points with 3 different layers  belongs to the layer
equivalence class $L\{a,b,c\}$ whenever $|L\{1, S\}|=a$, $|L\{2, S\}|=b$, and
$|L\{3, S\}|=c$, where $a, b, c$ are positive integers.

It is known that there exist sets with 9 points without any 5-hole, belonging to the layer equivalence classes $L\{3, 3, 3\}$ \cite{krasser} and $L\{3, 5, 1\}$
\cite{harborth}. In the following theorem we show that any 9-point
set not belonging to either of these two equivalent classes contains a 5-hole.

\begin{theorem}
Any set of 9 points in the plane in general position, not containing a 5-hole
either belongs to the layer equivalence class $L\{3, 3, 3\}$ or to the layer
equivalence class $L\{3, 5, 1\}$. \label{th:layer}
\end{theorem}
\begin{proof}Let $S$ be a set of 9 points in general position. If $|\bbV(CH(S))|\geq 4$, a 5-hole is
guaranteed from Corollary \ref{corollary:nine_points}. Thus, for proving the
result is suffices to show that $S$ contains a 5-hole if $S\in L\{3, 4, 2\}$.

Assume $S \in L\{3, 4, 2\}$ and suppose $z_1, z_2, z_3, z_4$ are the vertices
of the second layer. Let $L\{3,
S\}=\{x, y\}$. The extended straight line ${xy}$ divides the entire plane into
two halfplanes. If one these halfplane contains three points of $L\{2, S\}$,
these three points along with the points $x$ and $y$ form a 5-hole.

Otherwise, both halfplanes induced by the extended straight line $xy$
contain exactly two points of $L\{2, S\}$. The exterior of the quadrilateral
$z_1z_2z_3z_4$ can now be partitioned into 4 disjoint regions $R_1$, $R_2$,
$R_3$, and $R_4$, as shown in Figure \ref{fig:fig34}(c). Let $|R_i|$ denote the
number of points of $\bbV(CH(S))$ in the region $R_i$. If $R_1$ or $R_3$
contains any point of $\bbV(CH(S))$, a 5-hole is immediate. Therefore,
$|R_1|=|R_3|=0$, which implies that $|R_2|+|R_4|=|\bbV(CH(S))|=3$. By the
pigeonhole principle, either $|R_2|\geq 2$ or $|R_4|\geq 2$. If $|R_2| \geq 2$,
$(R_2\cap S)\cup\{x, z_1, z_2\}$ contains a 5-hole. Otherwise, $|R_4|\geq 2$,
and $(R_4\cap S)\cup\{y, z_3, z_4\}$ contains a 5-hole.

Thus, a set $S$ of 9 points not containing a 5-hole, must either belong to
$L\{3, 3, 3\}$ or $L\{3, 5, 1\}$. \hfill $\Box$
\end{proof}

\section{Disjoint 5-Holes: Lower Bounds}
\label{mainresults}

In this section we present our main results concerning the existence of
disjoint 5-holes in planar point sets, which leads to a non-trivial lower bound
on the number of disjoint 5-holes in planar point sets. As $H(5)=10$, it is clear that every
set 20 points in the plane in general position, contains two disjoint 5-holes.
At first, we improve upon this result by showing that any set
of 19 points also contains two disjoint 5-holes.

\begin{theorem}Every set of 19 points in the plane in general position, contains two disjoint 5-holes.
\label{th:nineteen}
\end{theorem}

Drawing parallel from the $2m+4$-point result for disjoint 4-holes due to
Hosono and Urabe \cite{urabecgta}, we prove a partitioning theorem for disjoint
5-holes for any set of $2m+9$ points in the plane in general position.

\begin{theorem}
For any set of $2m+9$ points in the plane in general position, it is possible
to divide the plane into three disjoint convex regions such that one contains a
set of 9 points which contains a 5-hole, and the others contain $m$ points
each, where $m$ is a positive integer. \label{th:twomplusnine}
\end{theorem}

Since $H(5)=10$, the trivial lower bound on $F_5(n)$ is
$\lfloor\frac{n}{10}\rfloor$. Observe that any set of 47 points can be
partitioned into two sets of 19 points each, and another set of 9 points
containing a 5-hole, by Theorem \ref{th:twomplusnine}. Hence, from Theorems
\ref{th:nineteen} and \ref{th:twomplusnine}, it follows that, $F_5(47)=5$.
Using this result, we obtain an improved lower bound on $F_5(n)$.

\begin{theorem}
$F_5(n)\geq\lfloor\frac{5n}{47}\rfloor$. \label{th:th1}
\end{theorem}

\begin{proof}Let $S$ be a set of $n$ points in the plane, no three of which are collinear.
By a horizontal sweep, we can divide the plane into $\lceil\frac{n}{47}\rceil$
disjoint strips, of which $\lfloor\frac{n}{47}\rfloor$ contain 47 points each
and one remaining strip $R$, with $|R|<47$. The strips having 47 points
contain at least 5 disjoint 5-holes, since $F_5(47)=5$ (Theorems
\ref{th:nineteen} and \ref{th:twomplusnine}). If $9k+1\leq |R|\leq 9k+9$, for
$k=0$ or $k=1$, there exist at least $k$ disjoint 5-holes in $R$. If $19\leq
|R|\leq 28$, Theorem \ref{th:nineteen} guarantees the existence of 2 disjoint
5-holes in $R$. Finally, if $9k+2\leq |R|\leq 9k+10$, for $k=3$ or $4$, at
least $k$ disjoint 5-holes exist in $R$. Thus, the total number of disjoint
5-holes in a set of $n$ points is always at least
$\lfloor\frac{5n}{47}\rfloor$. \hfill $\Box$
\end{proof}

We can obtain a better lower bound on $F_5(n)$ for infinitely many $n$, of the
form $n=28 \cdot 2^{k-1}-9$ with $k\geq 1$, by the repeated application of Theorem
\ref{th:twomplusnine}.

\begin{theorem}
$F_5(n)\geq (3n-1)/28$, for $n=28\cdot2^{k-1}-9$ and $k\geq 1$. \label{th:th2}
\end{theorem}
\begin{proof}Let $g(k)=28\cdot2^{k-1}-9$ and $h(k)=3\cdot 2^{k-1}-1$.
We need to show $F_5(g(k))\geq h(k)$. We prove the inequality by induction on
$k$. By Theorem \ref{th:nineteen}, the inequality holds for $k=1$. Suppose the
result is true for $k$, that is, $F_5(g(k))\geq h(k)$. Since, $g(k+1)=2g(k)+9$,
any set of $g(k+1)$ points can be partitioned into three disjoint convex
regions, two of which contain $g(k)$ points each, and the third a set of 9
points containing a 5-hole by Theorem \ref{th:twomplusnine}. Hence,
$F_5(g(k+1))$ $=F_5(2g(k)+9)\geq 2h(k)+1=h(k+1)$. This completes the induction
step, proving the result for $n=28\cdot2^{k-1}-9$. \hfill $\Box$
\end{proof}

\section{Proof of Theorem \ref{th:nineteen}}
\label{proofnineteen}

Let $S$ be a set of 19 points in the plane in general position. We say $S$ is
{\it admissible} if it contains two disjoint 5-holes. We prove Theorem
\ref{th:nineteen} by considering the various cases based on the size of
$|\bbV(CH(S))|$. The proof is divided into two subsections. The first section
considers the cases where $|\bbV(CH(S))|\geq 4$, and the second section deals with
the case where $|\bbV(CH(S))|=3$.

\subsection{$|\bbV(CH(S))|\geq 4$}

Let $CH(S)$ be the polygon $s_1s_2\ldots s_k$, where $k=|\bbV(CH(S))|$ and $k\geq 4$. A diagonal $d:=s_is_j$ of
$CH(S)$, is called a {\it dividing} diagonal if
$$\left||\mathcal H(s_is_j, s_m)\cap \bbV(CH(S))|-|\overline \mathcal H(s_is_j, s_m)\cap \bbV(CH(S))|\right|=c,$$
where $c$ is 0 or 1 according as $k$ is even or odd, and $s_m\in \bbV(CH(S))$
is such that $m\ne i, j$. Consider a dividing diagonal $d:=s_is_j$ of $CH(S)$.
Observe that for any fixed index $m\ne i, j$, either $|\mathcal H(s_is_j,
s_m)\cap S|\geq 9$ or $|\overline \mathcal H(s_is_j, s_m)\cap S|\geq 9$. Now,
we have the following observation.

\begin{observation}If for some dividing diagonal $d=s_is_j$ of $CH(S)$, $|\mathcal H(s_is_j, s_m)\cap S|>10$, where $m\ne i, j$, then $S$ is admissible.
\label{ob:10+t}
\end{observation}
\begin{proof}Let $Z=\overline \mathcal H_c(s_is_j, s_m)\cap S$ and $\beta$ and $\gamma$ the first and the second angular neighbors of $\overrightarrow{s_is_j}$ in $\mathcal H(s_is_j, s_m)\cap S$, respectively. Now, $|\bbV(CH(Z))|\geq 3$, since $|\bbV(CH(S))|> 3$. We consider different cases based on the size of $CH(Z)$.

\begin{description}
\item[{\it Case} 1:] $|\bbV(CH(Z))|\geq 5$. This implies that $|\bbV(CH(Z \cup\{\beta\}))|\geq 6$ and so
$Z \cup\{\beta\}$ contains a 5-hole by Lemma \ref{lm:lm1}. This 5-hole is
disjoint from the 5-hole contained in $(\mathcal H(s_is_j, s_m)\cap
S)\backslash \{\beta\}$.

\item[{\it Case} 2:] $|\bbV(CH(Z))|=4$. If $|\bbI (CH(Z))|\geq 2$, then
$Z\cup\{\beta\}$ is a convex pentagon with at least two interior points. From
Lemma \ref{lm:lm2}, $Z\cup\{\beta\}$ contains a 5-hole which is disjoint from
the 5-hole contained in $(\mathcal H(s_is_j, s_m)\cap S)\backslash \{\beta\}$.
Otherwise, $|\bbI(CH(Z))|\leq 1$. Let $Z'=Z\cup\{\beta, \gamma\}$. It follows
from Lemmas \ref{lm:lm1} and \ref{lm:lm2} that $Z'$ always contains a 5-hole.
This 5-hole is disjoint from the 5-hole contained in $(\mathcal H(s_is_j,
s_m)\cap S)\backslash\{\beta, \gamma\}$, since $|(\mathcal H(s_is_j, s_m)\cap
S)\backslash\{\beta, \gamma\}|\geq 12$.

\item[{\it Case} 3:]$|\bbV(CH(Z))|=3$. If $|\bbI (CH(Z))|=5$,
$|\bbV(CH(Z\cup\{\beta\}))|= 4$ and $Z\cup\{\beta\}$ contains a 5-hole by Corollary
\ref{corollary:nine_points}, which is disjoint from the 5-hole contained in
$(\mathcal H(s_is_j, s_m)\cap S)\backslash\{\beta\}$. So, let $|\bbI
(CH(Z))|=b\leq 4$, which implies, $|\mathcal H(s_is_j, s_m)\cap S|=16-b$. Let
$\eta$ be the $(6-b)$-th angular neighbor of $\overrightarrow{s_is_j}$ in
$\mathcal H(s_is_j, s_m)\cap S$. Let $S_1=\mathcal H_c(\eta s_i, s_j)\cap S$
and $S_2=\overline\mathcal H(\eta s_i, s_j)\cap S$. Now, since $|S_1|=9$ and
$|\bbV(CH(S_1))|\geq 4$, $S_1$ contains 5-hole, by Corollary
\ref{corollary:nine_points}. This 5-hole disjoint from the 5-hole contained in
$S_2$. \hfill $\Box$
\end{description}
\end{proof}

Observation \ref{ob:10+t} implies that for any dividing diagonal $d:=s_is_j$
and for any fixed vertex $s_m$, with $m\ne i, j$, $S$ is admissible unless
$|\mathcal H(s_is_j, s_m)\cap S|\leq 10$ and $|\overline\mathcal H(s_is_j,
s_m)\cap S|\leq 10$. This can now be used to show the admissibility of $S$
whenever $|\bbV(CH(S))|\geq 8$.

\begin{lemma}
$S$ is admissible whenever $|\bbV(CH(S))|\geq 8$. \label{lm:ch8}
\end{lemma}

\begin{proof}Let $d:=s_is_j$ be a dividing diagonal of $CH(S)$, and $s_m\in \bbV(CH(S))$ be such that $m\ne i, j$.
Since $|\bbV(CH(S))|\geq 8$, both $|\mathcal H(s_is_j, s_m)\cap \bbV(CH(S))|$ and
$|\overline\mathcal H(s_is_j, s_m)\cap \bbV(CH(S))|$  must be greater than 3.
Moreover, if $|\mathcal H(s_is_j, s_m)\cap S|>10$, Observation \ref{ob:10+t}
ensures that $S$ is admissible. Thus, we have the following two cases:

\begin{description}
\item[{\it Case} 1:]$|\mathcal H(s_is_j, s_m)\cap S|=10$. Now, since
$|\bbV(CH(\overline \mathcal H_c(s_is_j, s_m)\cap S))|\geq 4$, $\overline \mathcal
H_c(s_is_j, s_m)\cap S$ contains a 5-hole which is disjoint from the 5-hole
contained in $\mathcal H(s_is_j, s_m)\cap S$.

\item[{\it Case} 2:]$|\mathcal H(s_is_j, s_m)\cap S|=9$. As $|\bbV(CH(S))|\geq 8$ and
$\overrightarrow{s_is_j}$ is a dividing diagonal of $CH(S)$, we have
$|\overline\mathcal H(s_is_j, s_m)\cap \bbV(CH(S))|\geq 3$. Let $W=(\overline
\mathcal H(s_is_j, s_m)\cap S)\cup \{s_i\}$. Then from Corollary
\ref{corollary:nine_points}, $W$ contains a 5-hole, since $|W|=9$ and
$|\bbV(CH(W))|\geq 4$. The 5-hole contained in $W$ is disjoint from the 5-hole
contained in $(\mathcal H(s_is_j, s_m)\cap S)\cup\{s_j\}$. Hence $S$ is
admissible.

\item[{\it Case} 3:]$|\mathcal H(s_is_j, s_m)\cap S|\leq 8$. In this case, $|\overline\mathcal H(s_is_j, s_m)\cap S|\geq 9$, and the problem reduces to the previous cases. \hfill $\Box$
\end{description}
\end{proof}

Therefore, it suffices to show the admissibility of $S$ whenever $4\leq
|\bbV(CH(S))|\leq 7$. Observe that $S$ is admissible whenever $|\mathcal H(s_is_j,
s_m)\cap S|=10$ and $|\bbV(CH(\overline \mathcal H_c(s_is_j, s_m)\cap S))| \geq 4$.
Moreover, {\it Case} 2 of Lemma \ref{lm:ch8} shows that $S$ is admissible if
$|\mathcal H(s_is_j, s_m)\cap S|=9$ and $|\overline\mathcal H(s_is_j, s_m)\cap
\bbV(CH(S))|\geq 3$. Thus, hereafter we shall assume,

\begin{assumption}
For every dividing diagonal $s_is_j$ of $CH(S)$, there exists $s_m\in \bbV(CH(S))$, with
$m\ne i, j$, such that either $|\mathcal H(s_is_j, s_m)\cap S|=10$ and $|\bbV(CH(\overline
\mathcal H_c(s_is_j, s_m)\cap S))|=3$, or $|\mathcal H(s_is_j, s_m)\cap S|=9$
and $|\overline\mathcal H(s_is_j, s_m)\cap \bbV(CH(S))|\leq 2$.
\label{assumption2}
\end{assumption}

A dividing diagonal $s_is_j$ of $CH(S)$ is said to be an $(a, b)-splitter$ of
$CH(S)$, where $a\leq b$ are integers, if either $|\mathcal H(s_is_j, s_m)\cap
S\backslash\bbV(CH(S))|=a$ and $|\overline\mathcal H(s_is_j, s_m)\cap
S\backslash\bbV(CH(S))|=b$ or $|\mathcal H(s_is_j, s_m)\cap
S\backslash\bbV(CH(S))|=b$ and $|\overline\mathcal H(s_is_j, s_m)\cap
S\backslash\bbV(CH(S))|=a$.

The admissibility of $S$ in the different cases which arise are now proved as
follows:

\begin{figure*}[h]
\centering
\begin{minipage}[c]{0.33\textwidth}
\centering
\includegraphics[width=1.9in]
    {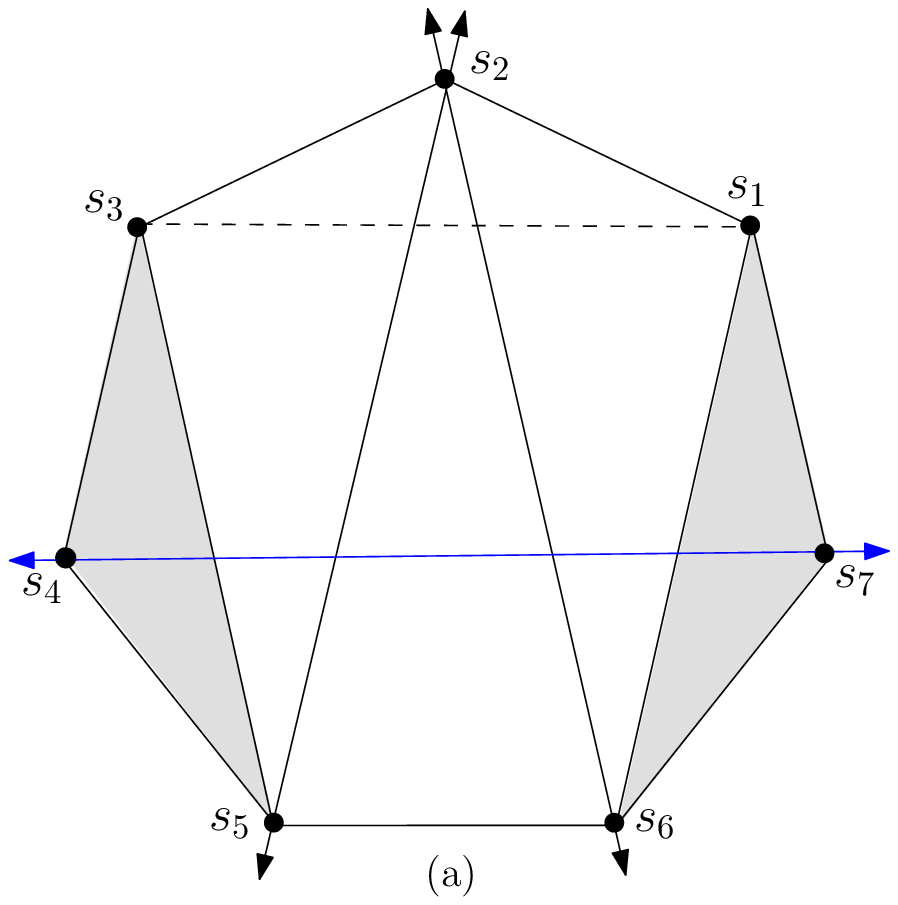}\\
\end{minipage}%
\begin{minipage}[c]{0.33\textwidth}
\centering
\includegraphics[width=1.9in]
    {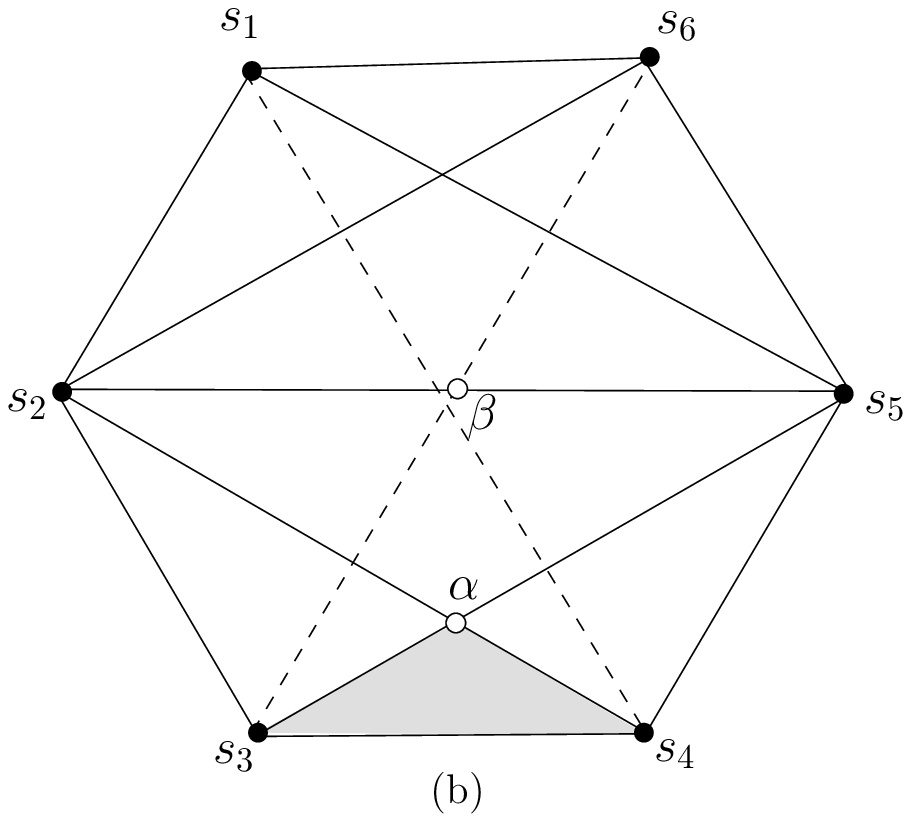}\\
\end{minipage}
\begin{minipage}[c]{0.33\textwidth}
\centering
\includegraphics[width=1.9in]
    {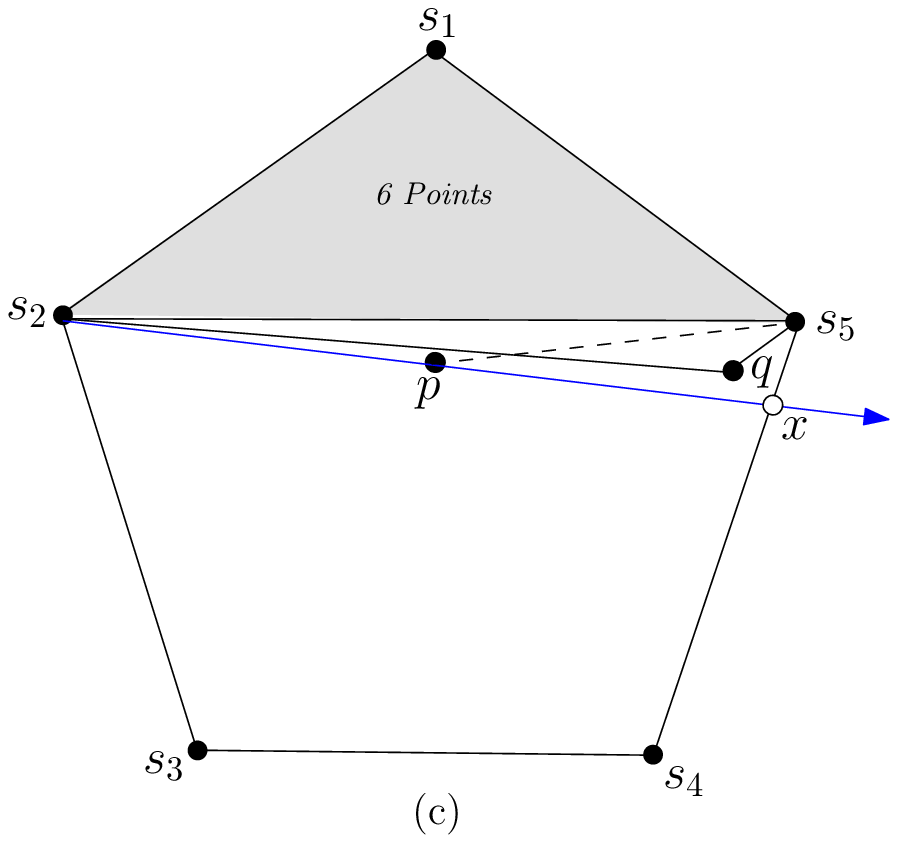}\\
\end{minipage}
\caption{Illustrations for the proof of Lemma \ref{lm:ch76}: (a) $|\bbV(CH(S))|=7$, (b) $|\bbV(CH(S))|=6$, (c) Illustration for the proof of Lemma \ref{lm:ch5}.}
  \label{fig:partII}\vspace{-0.15in}
\end{figure*}

\begin{lemma}$S$ is admissible whenever $6\leq |\bbV(CH(S))|\leq 7$.
\label{lm:ch76}
\end{lemma}

\begin{proof}We consider the two cases based on the size of $|\bbV(CH(S))|$ separately as follows:

\begin{description}
\item[{\it Case} 1:] $|\bbV(CH(S))|=7$. Refer to Figure \ref{fig:partII}(a). From Assumption \ref{assumption2} it follows that every dividing diagonal of
$CH(S)$ must be a $(6, 6)$-splitter of $CH(S)$. As both $s_2s_5$ and $s_2s_6$ are $(6,6)$-splitters, it is clear that $\bbI(s_2s_5s_6)$ is empty in $S$. Now, if $s_2$ is 5-redundant in either $\mathcal H_c(s_2s_5, s_4)\cap S$ or $\mathcal H_c(s_2s_6, s_2)\cap S$, the admissibility of $S$ is immediate. Therefore, assume that $s_2$ is not 5-redundant in either $\mathcal H_c(s_2s_5, s_4)\cap S$ or $\mathcal H_c(s_2s_6, s_2)\cap S$. This implies that
$\bbI(s_2s_3s_4s_5)\cap S \subset \bbI(s_3s_4s_5)$ and $\bbI(s_2s_6s_1s_7)\cap S \subset \bbI(s_1s_6s_7)$. Therefore, $\bbI( s_1s_2s_3)$ is empty in $S$. Now, since $s_4s_7$ is also a $(6,
6)$-splitter of $CH(S)$, $|\bbV(CH(\mathcal H(s_4s_7, s_2)\cap S))|\geq 4$ (see
Figure \ref{fig:partII}(a)), and Corollary \ref{corollary:nine_points} implies
$\mathcal H(s_4s_7, s_2)\cap S$ contains a 5-hole. This 5-hole disjoint from
the 5-hole contained in $\mathcal H_c(s_4s_7, s_5)\cap S$.

\item[{\it Case} 2:] $|\bbV(CH(S))|=6$. Refer to Figure \ref{fig:partII}(b).
Again, Assumption \ref{assumption2} implies that every dividing diagonal of
$CH(S)$ must be a $(6, 7)$-splitter of $CH(S)$. W.l.o.g. assume that $|\bbI
(s_1s_2s_5s_6)\cap S|=7$ and $|\bbI (s_2s_3s_4s_5)\cap S|=6$. Let $\alpha$ be
the point of intersection of the diagonals of the quadrilateral $s_2s_3s_4s_5$.
If $s_2$ or $s_5$ is 5-redundant in $\mathcal H_c(s_2s_5, s_4)\cap S$, then the admissibility of $S$ is immediate.
Therefore, assume that neither $s_2$ nor $s_5$ is 5-redundant in $\mathcal H_c(s_2s_5, s_4)\cap S$.
This implies that $\bbI (s_2s_3s_4s_5)\cap S\subset \bbI( s_3\alpha s_4)$. Now, if $|\bbI (s_1s_2s_3s_4)\cap S|=6$, then $s_4$ is 5-redundant in $\mathcal H_c(s_1s_4, s_2)\cap S$ and the admissibility of $S$ follows. Similarly, if $|\bbI(s_3s_4s_5s_6)\cap S|=6$, then $S$ is admissible, as $s_3$ is 5-redundant in $\mathcal H_c(s_3s_6, s_5)\cap S$. Hence, assume $|\bbI (s_1s_2s_3s_4)\cap S|=|\bbI(s_3s_4s_5s_6)\cap
S|=7$. Now, as $|\bbI(s_2s_3s_4s_5)\cap S|=6$, $(\bbI(s_3s_4s_5s_6)\backslash \bbI(s_3s_4\alpha))\cap S\subset \bbI(s_5s_6\beta)$, where $\beta$ is the point of intersection of the diagonals $s_2s_5$ and $s_3s_6$. Therefore,
$|\bbV(CH(\mathcal H(s_3s_6, s_5)\cap S))|\geq
4$. Therefore, the 5-hole contained in $\mathcal
H(s_3s_6, s_5)\cap S$ is disjoint from the 5-hole contained in $\mathcal
H_c(s_3s_6, s_1)\cap S$. \hfill $\Box$
\end{description}
\end{proof}

\begin{lemma}$S$ is admissible whenever $|\bbV(CH(S))|=5$.
\label{lm:ch5}
\end{lemma}

\begin{proof}Assumption \ref{assumption2} implies that a dividing diagonal of $CH(S)$ is either
a $(6, 8)$-splitter or a $(7, 7)$-splitter of $CH(S)$. To begin with suppose,
every dividing diagonal of $CH(S)$ is a $(7, 7)$-splitter of $|\bbV(CH(S))|$. Then
$|\bbI ( s_1s_2s_3)\cap S|=$ $|\bbI ( s_1s_4s_5)\cap S|=7$, which means that
$|\bbI ( s_1s_3s_4)\cap S|=0$. Similarly, $|\bbI ( s_2s_4s_5)\cap S|=$ $|\bbI (
s_3s_5s_1)\cap S|=$ $|\bbI ( s_4s_2s_1)\cap S|=$ $|\bbI ( s_5s_2s_3)\cap S|=0$.
This implies $|\bbI(CH(S))|=0$, which is a contradiction.

Therefore, assume that there exists a  $(6, 8)$-splitter of $CH(S)$. W.l.o.g., assume $s_2s_5$ is a $(6, 8)$-splitter of $CH(S)$. There are two
possibilities:
\begin{description}
\item[{\it Case} 1:] $|\bbI ( s_1s_2s_5)\cap S|=6$ and
$|\bbI (s_2s_3s_4s_5)\cap S|=8$. Refer to Figure \ref{fig:partII}(c). Let $p$
be the nearest neighbor of $s_2s_5$ in $\mathcal H(s_2s_5, s_4)\cap S$. W.l.o.g.,
assume $\bbI( s_1s_2p)\cap S$ is non-empty. Let $x$ be the point where
$\overrightarrow{s_2p}$ intersects the boundary of $CH(S)$. Then $\mathcal
H_c(s_2x, s_1)\cap S$ contains a 5-hole, and by Corollary \ref{cor:5redun} $s_2$
is 5-redundant in $\mathcal H_c(s_2p, s_1)\cap S$. Now, if $Cone(s_5px)\cap S$
is empty, the 5-hole contained in $(\mathcal H_c(s_2p, s_1)\cap
S)\backslash\{s_2\}$ is disjoint from the 5-hole contained in
$(\overline\mathcal H(s_2p, s_1)\cap S)\cup\{s_2\}$. Otherwise, assume
$Cone(s_5px)\cap S$ is non-empty. Let $q$ be the first angular neighbor of
$\overrightarrow {s_2s_5}$ in $Cone(s_5px)$. Observe that $\bbI( s_1s_2q)\cap
S$ is non-empty, since $\bbI( s_1s_2p)\cap S$ is assumed to be non-empty, and
$\mathcal H_c(s_2q, s_1)\cap S$ contains a 5-hole. Now, Corollary
\ref{cor:5redun} implies that $s_2$ is 5-redundant in $\mathcal H_c(s_2q,
s_1)\cap S$, and the admissibility of $S$ follows.

\item[{\it Case} 2:]$|\bbI ( s_1s_2s_5)\cap S|=8$ and
$|\bbI (s_2s_3s_4s_5)\cap S|=6$. Clearly, $\mathcal H_c(s_2s_5, s_3)\cap S$
contains a 5-hole. Now, if either $s_2$ or $s_5$ is 5-redundant in $\mathcal
H_c(s_2s_5, s_3)\cap S$, then $S$ is admissible. Therefore, assume $\bbI
(s_2s_3s_4s_5)\cap S\subset \bbI( s_3s_4\alpha)$, where $\alpha$ is the point
where the diagonals of the quadrilateral $s_2s_3s_4s_5$ intersect. The
problem now reduces to {\it Case} 1 with  respect to the dividing diagonal
$s_2s_4$. \hfill $\Box$
\end{description}
\end{proof}

The case $|\bbV(CH(S))|=4$ is dealt separately in the next section.

\subsubsection{$|\bbV(CH(S))|=4$}

As before, let $CH(S)$ be the polygon $s_1s_2s_3s_4$. From Observation \ref{ob:10+t}, we have to
consider the cases where a dividing diagonal of $CH(S)$ is either a $(6,
9)$-splitter or a $(7, 8)$-splitter of $CH(S)$.

Firstly, suppose some dividing diagonal of $CH(S)$, say $s_2s_4$, is a $(6,
9)$-splitter of $CH(S)$. Assume that $|\bbI ( s_1s_2s_4)\cap S|=6$ and $|\bbI (
s_2s_3s_4)\cap S|=9$. Begin by taking the nearest neighbor $p$ of $s_2s_4$ in
$\bbI(s_2s_3s_4)$. Then choose the first angular neighbor $q$ of either
$\overrightarrow{s_2s_4}$ or $\overrightarrow{s_4s_2}$ in $\bbI(s_2s_3s_4)$,
and proceed as in {\it Case} 1 of Lemma \ref{lm:ch5} to show the admissibility
of $S$.

Therefore, it suffices to assume that

\begin{assumption}Both the dividing diagonals of the quadrilateral $s_1s_2s_3s_4$ are $(7, 8)$-splitters of $CH(S)$.
\label{assumption_ch4_1}
\end{assumption}

W.l.o.g., let $|\bbI ( s_1s_2s_4)\cap S|=8$ and $|\bbI ( s_2s_3s_4)\cap
S|=7$. Let $\alpha$ be the point where the diagonals of $CH(S)$ intersect.
Observe, there always exists an edge of $CH(S)$ say, $s_2s_3$, such that
$|\bbI( s_1s_2s_3)\cap S|=$ $|\bbI ( s_2s_3s_4)\cap S|=7$, and $|\bbI (
s_1s_3s_4)\cap S|=$ $|\bbI ( s_1s_2s_4)\cap S|=8$. This implies, $|\bbI (
s_1s_2\alpha)\cap S|=|\bbI ( s_3s_4\alpha)\cap S|=n$, with $0\leq n\leq 7$. We
begin with the following simple observation

\begin{lemma}$S$ is admissible whenever $n=0$.
\label{ob:ch4_nzero}
\end{lemma}

\begin{proof}Let $Z=(\mathcal H(s_2s_4, s_1)\cap S)\cup\{s_4\}$. Observe that $|Z|=10$, which means $Z$ contains a 5-hole. If $|\bbV(CH(Z))|\geq 5$, $s_4$ is 5-redundant in $Z$, and $Z\backslash \{s_4\}$ contains a 5-hole which is disjoint from the 5-hole contained in $\mathcal H_c(s_2s_4, s_3)\cap S$. Let $r$ be the nearest angular neighbor of $\overrightarrow{s_1s_3}$ in $Cone(s_4s_1s_3)$. If $|\bbV(CH(Z))|=4$, either $r$ or $s_4$ is 5-redundant in $Z$ by Corollary \ref{cor:5redun}, and the admissibility of $S$ follows. Otherwise, $|\bbV(CH(Z))|=3$ and at least one of $s_1$, $s_4$, or $r$ is 5-redundant in $Z$ and the admissibility of $S$ follows similarly. \hfill$\Box$
\end{proof}

From the previous lemma, it suffices to assume $n>0$. Let $p$ be the first
angular neighbor of $\overrightarrow{s_2s_4}$ in $Cone(s_4s_2s_3)$ and $x$ the
intersection point of $\overrightarrow{s_2p}$ with the boundary of $CH(S)$. Let
$\alpha$ be the point of intersection of the diagonals of the quadrilateral
$s_1s_2s_3s_4$. If $Cone(s_3px)\cap S$ is non-empty, $|\bbV(CH(\mathcal H_c(s_2p,
s_3)\cap S))|\geq 4$. From Corollary \ref{corollary:nine_points}, $\mathcal
H_c(s_2p, s_3)\cap S$ contains a 5-hole which is disjoint from the 5-hole
contained in $(\mathcal H(s_2s_4, s_1)\cap S)\cup\{s_4\}$. Therefore, we shall
assume that

\begin{assumption}
$Cone(s_3px)\cap S$ is empty. \label{assumption_ch4_2}
\end{assumption}

Assumption \ref{assumption_ch4_2} and the fact that $n>0$ implies that $p\in
\bbI ( s_3\alpha s_4)\cap S$ (see Figure \ref{fig:hullq}(a)). Let $q$ be the
first angular neighbor of $\overrightarrow{ps_2}$ in $Cone(s_2ps_1)$. The
admissibility of $S$ in the remaining cases is proved in the following two
lemmas.

\begin{figure*}[h]
\centering
\begin{minipage}[c]{0.33\textwidth}
\centering
\includegraphics[width=2.0in]
    {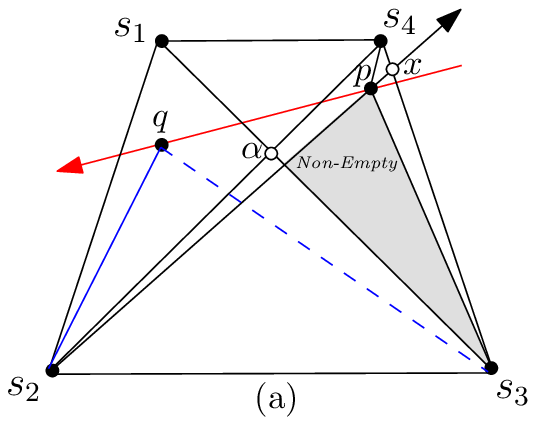}\\
\end{minipage}%
\begin{minipage}[c]{0.33\textwidth}
\centering
\includegraphics[width=2.0in]
    {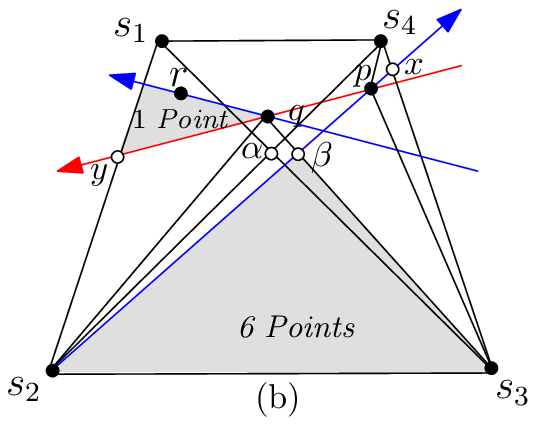}\\
\end{minipage}
\begin{minipage}[c]{0.33\textwidth}
\centering
\includegraphics[width=2.0in]
    {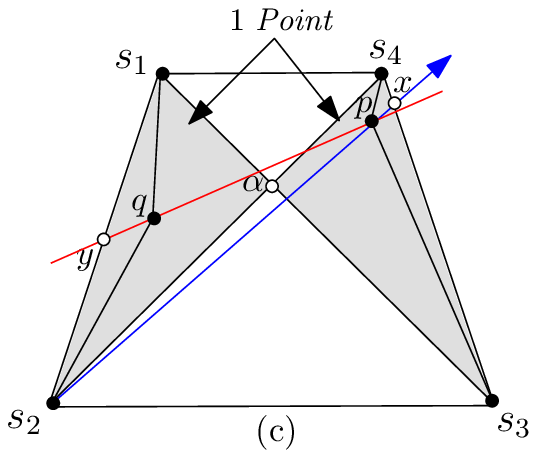}\\
\end{minipage}
\caption{$|\bbV(CH(S))|=4$: (a) $|\bbI ( s_1s_2\alpha)\cap S|=|\bbI
( s_3s_4\alpha)\cap S|=n\geq 2$ and $q\in \bbI ( s_2\alpha s_1)\cap S$, (b)
$|\bbI ( s_1s_2\alpha)\cap S|=|\bbI ( s_3s_4\alpha)\cap S|=n\geq 2$, and $q \in
\bbI( s_1\alpha s_4)\cap S$, (c) $|\bbI ( s_1s_2\alpha)\cap S|=|\bbI (
s_3s_4\alpha)\cap S|=n=1$.}
\label{fig:hullq}\vspace{-0.15in}
\end{figure*}

\begin{lemma}$S$ is admissible whenever $n\geq 2$.
\label{lm:ch4_ngeq2}
\end{lemma}

\begin{proof}
To begin with suppose, $q\in \bbI ( s_2\alpha s_1)\cap S$, as shown in Figure
\ref{fig:hullq}(a). By Assumption \ref{assumption_ch4_2}, there exists a point
in $\bbI(s_3s_4\alpha)\cap S$, different from the point $p$, which belongs to
$\bbI( qps_3)\cap S$. Hence, by Corollary \ref{cor:5redun}, $p$ is 5-redundant
in $\mathcal H_c(pq, s_2)\cap S$, and the 5-hole contained in $(\mathcal H(pq,
s_2)\cap S)\cup\{q\}$ is disjoint from the 5-hole contained in $(\mathcal H(pq,
s_1)\cap S)\cup\{p\}$.

Otherwise, assume that $q \in \bbI( s_1\alpha s_4)\cap S$ and refer to Figure
\ref{fig:hullq}(b). Observe that $S$ is admissible if either $p$ or $q$ is
5-redundant in $\mathcal H_c(pq, s_2)\cap S$. Hence, assume that neither $p$
nor $q$ is 5-redundant in $\mathcal H_c(pq, s_2)\cap S$. This implies
$\bbI(s_2s_3pq)\cap S\subset \bbI( s_2s_3\beta)$, where $\beta$ is the point of
intersection of the diagonals of the quadrilateral $ s_2s_3pq$. Let $r$ be the
second angular neighbor of $\overrightarrow{qy}$ in $Cone(yqs_1)$, where $y$ is
the point where $\overrightarrow{pq}$ intersects the boundary $CH(S)$. Note
that the point $r$ exists because $n\geq 2$ and $q\in \bbI( s_1s_4\alpha)\cap
S$. Now, the 5-hole contained in $(\mathcal H(qr, s_2)\cap S)\cup\{q\}$ is
disjoint from the 5-hole contained in $(\mathcal H(qr, s_1)\cap S)\cup\{r\}$ by
Corollary \ref{corollary:nine_points}. \hfill $\Box$
\end{proof}

\begin{lemma}$S$ is admissible whenever $n=1$.
\label{lm:ch4_n1}
\end{lemma}

\begin{proof}
To begin with let $q\in \bbI( s_1\alpha s_2)$. Refer to Figure \ref{fig:hullq}(c). Assume, $\bbI( s_4pq)\cap S$ is
non-empty and let $Z=(\mathcal H(pq, s_1)\cap S)\cup\{q\}$. Observe that
$|\bbV(CH(Z))|\geq 4$, and by Corollary \ref{cor:5redun} either $q$ or $s_4$ is
5-redundant in $Z$, and the admissibility of $S$ follows.

Otherwise, assume that $\bbI( s_4pq)\cap S$ is empty. If either $q$ or $s_4$ is
5-redundant in $Z$, the admissibility of $S$ is immediate. Therefore, it
suffices to assume that there exists a 5-hole in $Z$ with $qs_4$ as an edge.
This implies that we have a 6-hole with $ps_4$ and $pq$ as edges. Observe that
$s_1$ cannot be a vertex of this 6-hole. Hence, there exists a 5-hole with
$ps_4$ as an edge, which does not contain $s_1$ and $q$ as vertices. Thus,
$s_1$ and $q$ are 5-redundant in $\mathcal H_c(s_4q, s_1)\cap S$. This 5-hole
is disjoint from the 5-hole contained in $\mathcal H_c(s_1s_3, s_2)\cap S$.

\begin{figure*}[h]
\centering
\begin{minipage}[c]{0.5\textwidth}
\centering
\includegraphics[width=2.75in]
    {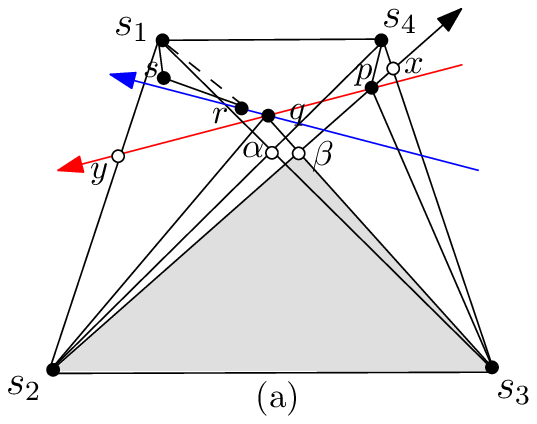}\\
\end{minipage}%
\begin{minipage}[c]{0.5\textwidth}
\centering
\includegraphics[width=2.75in]
    {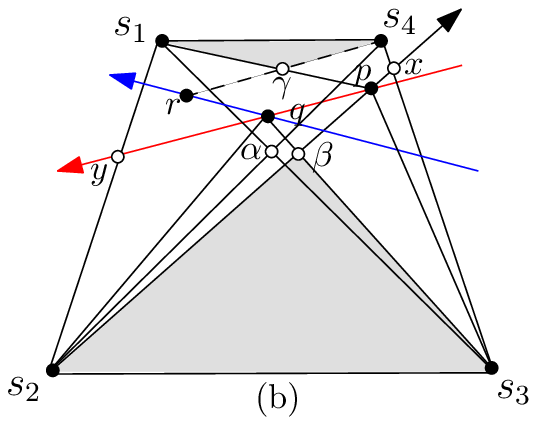}\\
\end{minipage}
\caption{$|\bbV(CH(S))|=4$ with $|\bbI (s_1s_2\alpha)\cap S|=|\bbI(s_3s_4\alpha)\cap
S|=n=1$: (a) $q, r \in \bbI( s_1s_4\alpha)\cap S$, and (b) $q\in
\bbI(s_1s_4\alpha)$ and $r\in \bbI(s_1s_2\alpha)$.} \vspace{-0.15in}
\label{fig:hullq_n=1}
\end{figure*}

Finally, suppose $q \in \bbI(s_1s_4\alpha)\cap S$ (see Figure
\ref{fig:hullq_n=1}(a)). Observe that since $Cone(s_3px)\cap S$ is empty by
Assumption \ref{assumption_ch4_2}, $S$ is admissible whenever either $p$ or $q$
is 5-redundant in $\mathcal H_c(pq, s_2)\cap S$. Hence, assume that
$\bbI(s_2s_3pq)\cap S\subset \bbI( s_2s_3\beta)$, where $\beta$ is the point of
intersection of the diagonals of the quadrilateral $ s_2s_3pq$. Let $r$ be the
first angular neighbor of $\overrightarrow{qy}$ in $Cone(yqs_1)$, where $y$ is
the point where $\overrightarrow{pq}$ intersects the boundary $CH(S)$. If $r\in
\bbI(s_1s_4\alpha)\cap S$, then $|\bbV(CH(\mathcal H_c(pq, s_1)\cap S))|=6$ and both
$p$ and $q$ are 5-redundant in $\mathcal H_c(pq, s_1)\cap S$ (Figure
\ref{fig:hullq_n=1}(a)). Thus, the partition of $S$ given by $\mathcal H(pq,
s_1)\cap S$ and $\mathcal H_c(pq, s_2)\cap S$ is admissible. Otherwise, assume
that $r\in \bbI(s_1s_2\alpha)\cap S$, as shown in Figure
\ref{fig:hullq_n=1}(b). Let $\gamma$ be the point of intersection of the
diagonals of the quadrilateral $s_1rps_4$. From Corollary \ref{cor:5redun}, it
is easy to see that whenever there exists a point of $(\mathcal H(pq, s_1)\cap
\bbI(s_1s_4\alpha))\cap S$ outside $\bbI(s_1s_4\gamma)$, at least one of $p$ or
$r$ is 5-redundant in $(\mathcal H(pq, s_1))\cap S\cup\{p\}$, and the
admissibility of $S$ is immediate. Therefore, it suffices to assume that
$(\mathcal H(pq, s_1)\cap \bbI(s_1s_4\alpha))\cap S\subset \bbI(s_1s_4\gamma)$.
Then $|\bbV(CH(\mathcal H(s_2s_4, s_1)\cap S))|\geq 4$ and $|\mathcal H(s_2s_4,
s_1)\cap S|=9$. Hence, the 5-hole contained in $\mathcal H(s_2s_4, s_1)\cap S$
(Corollary \ref{corollary:nine_points}), is disjoint from the 5-hole contained
in $\mathcal H_c(s_2s_4, s_3)\cap S$. \hfill $\Box$
\end{proof}

\subsection{$|\bbV(CH(S))|=3$}

Let $s_1, s_2, s_3$ be the three vertices of $CH(S)$. Let $\bbI(CH(S))=\{u_1, u_2, \dots, u_{16}\}$ be such that $u_i$ is
the $i$-th angular neighbor of $\overrightarrow{s_1s_2}$ in
$Cone(s_2s_1s_{3})$. For $i\in\{1, 2, 3\}$ and $j\in\{1, 2, \ldots, 16\}$, let
$p_{ij}$ be the point where $\overrightarrow{s_iu_j}$ intersects the boundary
of $CH(S)$. For example, $p_{17}$ is the point of intersection of
$\overrightarrow{s_1u_7}$ with the boundary of $CH(S)$.

If $\bbI( u_7p_{17}s_2)$ is not empty in $S$, $|\bbV(CH(\mathcal H_c(s_1u_7,
s_2)\cap S))|\geq 4$ and by Corollary \ref{corollary:nine_points}, $\mathcal
H_c(s_1u_7, s_2)\cap S$ contains a 5-hole which is disjoint from the 5-hole
contained in $\mathcal H(s_1u_7, s_3)\cap S$. Therefore, $\bbI(
u_7p_{17}s_2)\cap S$ can be assumed to be empty. In fact, we can make the
following more general assumption.

\begin{assumption}
For all $i\ne j\ne k\in\{1, 2, 3\}$, $Cone(p_{it}u_ts_j)\cap S$ is empty, where
$u_t$ is the seventh angular neighbor of $\overrightarrow{s_is_j}$ in
$Cone(s_js_is_k)\cap S$. \label{assumption3}
\end{assumption}

Now, we have the following observation.

\begin{observation}If for some $i\ne j\ne k\in\{1, 2, 3\}$, $Cone(p_{it}u_ts_j)\cap S$ is non-empty, where
$u_t$ is the eighth angular neighbor of $\overrightarrow{s_is_j}$ in
$Cone(s_js_is_k)$, then $S$ is admissible. \label{ob:obeighthalways}
\end{observation}

\begin{proof}W.l.o.g., let $i=1$ and $j=2$, which means, $t=8$.
Set $T=\mathcal H_c(s_1u_8, s_2)\cap S$. Suppose, there exists a point
$u_r\in\bbI( s_2u_{8}p_{18})\cap S$. This implies that $|\bbV(CH(T))|\geq 4$. When
$|\bbV(CH(T))|\geq 5$, $u_{8}$ is 5-redundant in $T$ and $T\backslash\{u_8\}$
contains a 5-hole which is disjoint from the 5-hole contained in $(\mathcal
H(s_1u_8, s_3)\cap S)\cup\{u_8\}$.

Hence, it suffices to assume $|\bbV(CH(T))|=4$. Let $\bbV (CH(T))=\{s_1, s_2, u_r,
u_{8}\}$, with $r\leq 7$, and $\alpha$ the point of intersection of the
diagonals of the quadrilateral $s_1s_2u_ru_{8}$. By Corollary \ref{cor:5redun},
it follows that unless $\bbI(s_1s_2u_ru_8)\cap S\subset \bbI( s_2\alpha u_r)$,
either $s_1$ or $u_{8}$ is 5-redundant in $T$ and hence $S$ is admissible.
Therefore, assume $\bbI(s_1s_2u_ru_8)\cap S\subset \bbI( s_2\alpha u_r)$, which
implies $u_r=u_7$, as shown in Figure \ref{fig:thull_new}(a). Suppose,
$Cone(s_1u_7u_8)\cap S$ is non-empty, and let $u_k$ be the first angular
neighbor of $\overrightarrow{u_7s_1}$ in $Cone(s_1u_7u_{8})$. Then
$\bbI(u_ku_7s_2)\cap S$ is non-empty, and $u_7$ is 5-redundant in $\mathcal
H_c(u_7u_k, s_1)\cap S$. Thus, the 5-hole contained in $\mathcal H(u_7u_k,
s_1)\cap S)\cup\{u_k\}$ is disjoint from the 5-hole contained in $(\mathcal
H(u_7u_k, s_3)\cap S)\cup\{u_7\}$. However, if $Cone(s_1u_7u_{8})\cap S$ is
empty, $u_7$ is 5-redundant in $\mathcal H_c(u_7u_8, s_1)\cap S$ by Corollary
\ref{cor:5redun}, and the 5-hole contained in $(\mathcal H(u_7u_8, s_1)\cap
S)\cup\{u_8\}$ is disjoint from the 5-hole contained in $(\mathcal H(u_7u_8,
s_3)\cap S)\cup\{u_7\}$. \hfill $\Box$
\end{proof}

\begin{figure*}[h]
\centering
\begin{minipage}[c]{0.33\textwidth}
\centering
\includegraphics[width=1.75in]
    {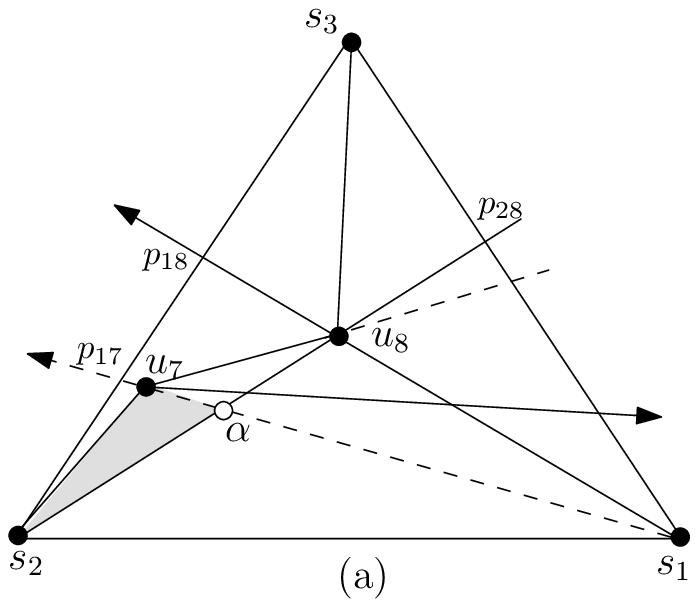}\\
\end{minipage}%
\begin{minipage}[c]{0.33\textwidth}
\centering
\includegraphics[width=1.75in]
    {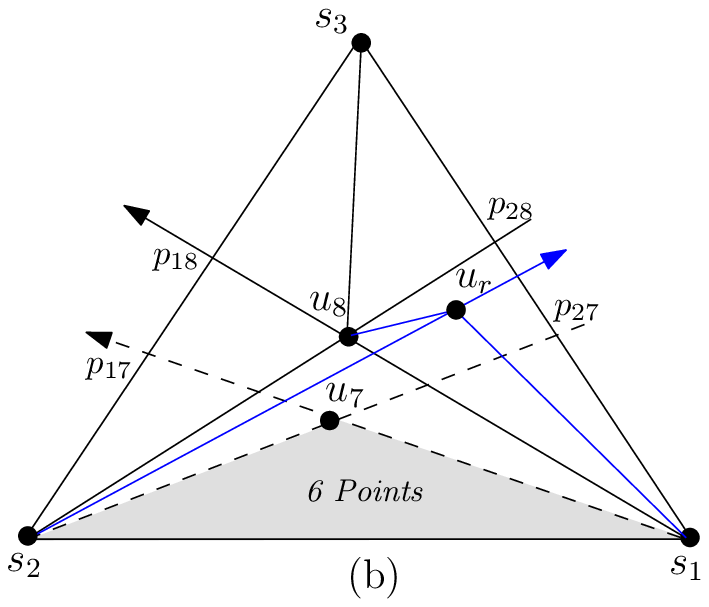}\\
\end{minipage}
\begin{minipage}[c]{0.33\textwidth}
\centering
\includegraphics[width=1.75in]
    {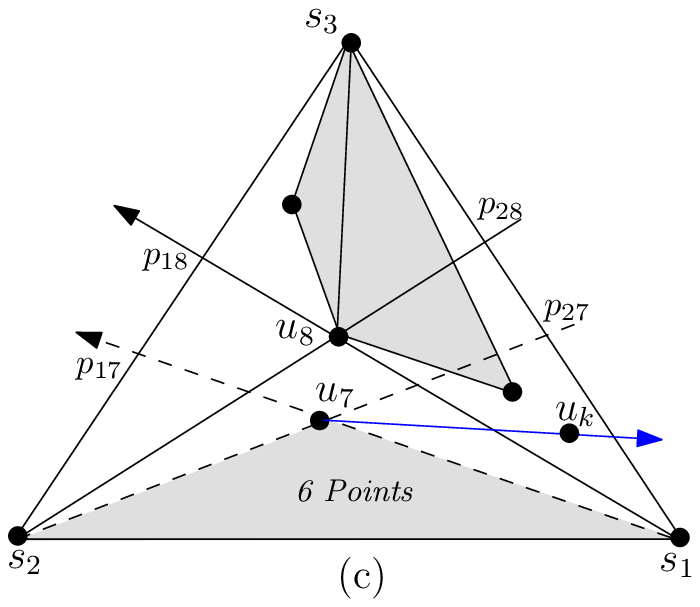}\\
\end{minipage}
\caption{(a) Proof of Observation \ref{ob:obeighthalways}, (b) Proof of Lemma \ref{lm:ch3empty}, and
  (c) Proof of Lemma \ref{lm:lmeightempty}.}
  \label{fig:thull_new}\vspace{-0.15in}
\end{figure*}

\begin{lemma}
If for some $i\ne j\ne k\in\{1, 2, 3\}$, $Cone(p_{jt}u_ts_i)\cap S$ is empty,
where $u_t$ is the seventh angular neighbor of $\overrightarrow{s_is_j}$ in
$Cone(s_js_is_k)$, then $S$ is admissible. \label{lm:ch3empty}
\end{lemma}

\begin{proof}
W.l.o.g., let $i=1$ and $j=2$. This means $t=7$ and $Cone(s_1u_7p_{27})$ is
empty in $S$. From Assumption \ref{assumption3}, $\bbI( u_7p_{17}s_2)\cap S$ is
empty. Based on Observation \ref{ob:obeighthalways} we may suppose $Cone(s_2u_8p_{18})\cap
S$ is empty. Now, if $Cone(p_{28}u_8s_1)\cap S$ is empty, at least one of
$s_1$, $s_2$, or $u_8$ is 5-redundant in $\mathcal H_c(s_1u_8, s_2)\cap S$, and
admissibility of $S$ is immediate.

Therefore, assume that $Cone(p_{28}u_8s_1)\cap S$ is non-empty, which implies
that $Cone(p_{27}s_2p_{28})\cap S$ is non-empty, since $Cone(s_1u_7p_{27})\cap
S$ is assumed to be empty. Let $u_r$ be the first angular neighbor of
$\overrightarrow{s_2u_7}$ in $Cone(p_{27}s_2p_{28})\cap S$ (see Figure
\ref{fig:thull_new}(b)). Now, $S$ is admissible unless there exists a 5-hole in
$\mathcal H_c(s_1u_8, s_2)\cap S$ with $s_1u_8$ as an edge. Observe that this
5-hole cannot have $s_2$ as a vertex. Moreover, the remaining three vertices of
this 5-hole, that is, the vertices apart from $s_1$ and $u_8$, lie in the
halfplane $\mathcal H(u_rs_2, s_1)$. Now, this 5-hole can be extended to a
convex hexagon having $s_1$, $u_8$, and $u_r$ as three consecutive vertices. Note that this convex
hexagon may not be empty, and it does not contain $s_2$ as a
vertex. From this convex hexagon, we can get a 5-hole with $u_rs_1$ as an edge,
which does not contain $u_8$ as a vertex and which lies in the halfplane
$\mathcal H(u_rs_1, s_2)$. Hence,  $(\mathcal H(s_2u_r, s_1)\cap S)\cup\{u_r\}$
contains a 5-hole which is disjoint from the 5-hole contained in $(\mathcal
H(s_2u_r, s_3)\cap S)\cup\{s_2\}$. \hfill $\Box$
\end{proof}

Hereafter, in light of the previous lemma, let us assume

\begin{assumption}
For all $i\ne j\ne k\in\{1, 2, 3\}$, $Cone(p_{jt}u_ts_i)\cap S$ is non-empty,
where $u_t$ is the seventh angular neighbor of $\overrightarrow{s_is_j}$ in
$Cone(s_js_is_k)$. \label{assumption4}
\end{assumption}

With this assumption we have the following two lemmas.

\begin{lemma}If for some $i\ne j\ne k\in\{1, 2, 3\}$, $Cone(s_ku_ts_j)\cap S$ is non-empty, where
$u_t$ is the eighth angular neighbor of $\overrightarrow{s_is_j}$ in
$Cone(s_js_is_k)\cap S$, then $S$ is admissible. \label{lm:lmeightempty}
\end{lemma}

\begin{proof}It suffices to prove the result for $i=1$ and $j=2$, which means $t=8$. Refer to Figure \ref{fig:thull_new}(c). Based on Observation \ref{ob:obeighthalways} we may suppose $S$ is admissible whenever
$\bbI(  s_2u_{8}p_{18})\cap S$ is non-empty. Therefore, assume that $\bbI(
s_2u_{8}p_{18})\cap S$ is empty. Now, suppose $\bbI( u_{8}s_{3}p_{18})\cap S$
is non-empty, and let $ \bbI( u_{8}s_{3}p_{18})\cap S$. Let $u_k$ be the
first angular neighbor of $\overrightarrow{u_7s_1}$ in $Cone(s_1u_7p_{27})$,
which is non-empty by Assumption \ref{assumption4}.
If $Cone(u_ku_7p_{27})$ is empty, from Corollary \ref{cor:5redun}, $s_2$ is
5-redundant in $\mathcal H_c(u_7u_k, s_2)\cap S$ and the admissibility of $S$
follows. Thus, there exists some point $u_m$ $(m\ne k)$ in
$Cone(u_ku_7p_{27})\cap S$. Therefore, $|\bbV(CH((\mathcal H(u_7u_k, s_3)\cap
S)))|\geq 4$, and by Corollary \ref{corollary:nine_points}, $\mathcal H(u_7u_k,
s_3)\cap S$ contains a 5-hole. This 5-hole is disjoint from the 5-hole
contained in $\mathcal H_c(u_7u_k, s_2)\cap S$.
\hfill $\Box$
\end{proof}

\begin{lemma}If for some $i\ne j\ne k\in\{1, 2, 3\}$, $Cone(s_ku_ts_j)\cap S$ is non-empty, where
$u_t$ is the seventh angular neighbor of $\overrightarrow{s_is_j}$ in
$Cone(s_js_is_k)$, then $S$ is admissible. \label{lm:lmsevenempty}
\end{lemma}

\begin{proof}W.l.o.g., let $i=1$ and $j=2$, which means $t=7$.
From Assumption \ref{assumption3}, $\bbI( u_7p_{17}s_2)\cap S$ is empty. Next,
suppose there exists a point $u_a$ in  $\bbI( u_7p_{17}s_{3})\cap S$. Refer to
Figure \ref{fig:thull}(a). Since $Cone(s_1u_7p_{27})\cap S$ is non-empty by
Assumption \ref{assumption4}, let $u_k$ be the first angular neighbor of
$\overrightarrow{u_7s_1}$ in $Cone(s_1u_7p_{27})$ and $\alpha$ the point of
intersection of the diagonals of the convex quadrilateral $u_7s_2s_1u_k$. From
Corollary \ref{cor:5redun}, it is easy to see that $S$ is admissible unless
$\bbI(s_1s_2u_7u_k)\cap S\subset \bbI(s_1s_2\alpha)$.
Now, if $u_7$ is the eighth angular neighbor of $\overrightarrow{s_2s_1}$ or
$\overrightarrow{s_2s_3}$ in $Cone(s_1s_2s_{3})$, then $S$ is admissible from
Lemma \ref{lm:lmeightempty}, since $\bbI( u_7s_3s_{1})\cap S$ is not empty.
Since the eighth angular neighbor of $\overrightarrow{s_2s_3}$ in $Cone(s_1s_2s_{3})$ is the ninth angular
neighbor of $\overrightarrow{s_2s_1}$ in $Cone(s_1s_2s_{3})$, $u_7$ cannot be the eighth or ninth angular neighbor
$\overrightarrow{s_2s_1}$ in $Cone(s_1s_2s_{3})$. Thus there exist at least
two points, $u_m$ and $u_n$ in $Cone(p_{27}u_7u_k)\cap S$, where $u_m$ is the first
angular neighbor of $\overrightarrow{u_7u_k}$ in $Cone(p_{27}u_7u_k)$. Then,
the 5-hole contained in $(\mathcal H(u_7u_m, s_1)\cap S)\cup\{u_m\}$ is
disjoint from the 5-hole contained in $(\mathcal H(u_7u_m, s_3)\cap
S)\cup\{u_7\}$, since $|\bbV(CH((\mathcal H(u_7u_m, s_3)\cap S)\cup\{u_7\}))|\geq 4$
(see Figure \ref{fig:thull}(a)).
\hfill $\Box$
\end{proof}

\begin{figure*}[h]
\centering
\begin{minipage}[c]{0.33\textwidth}
\centering
\includegraphics[width=1.75in]
    {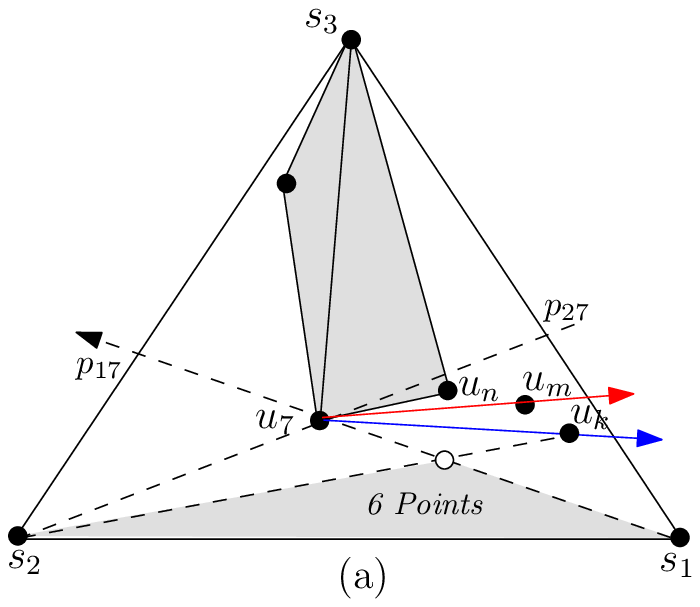}\\
\end{minipage}%
\begin{minipage}[c]{0.33\textwidth}
\centering
\includegraphics[width=1.75in]
    {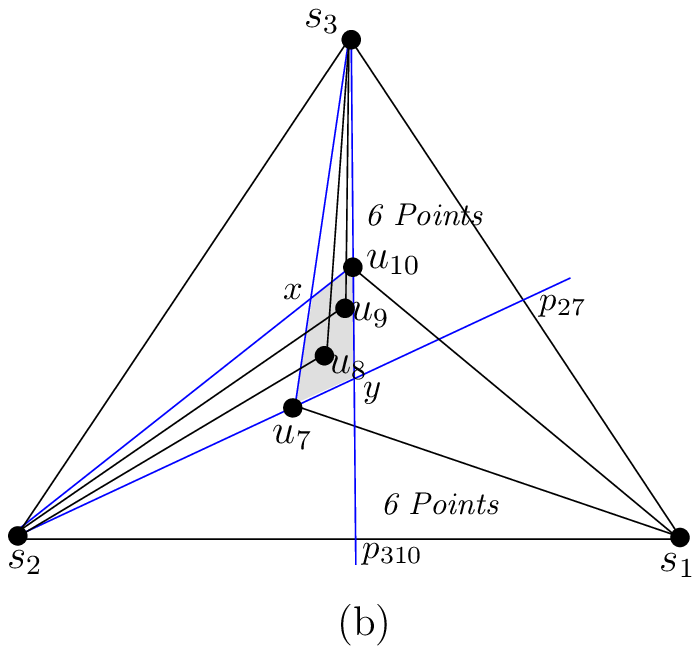}\\
\end{minipage}
\begin{minipage}[c]{0.33\textwidth}
\centering
\includegraphics[width=1.75in]
    {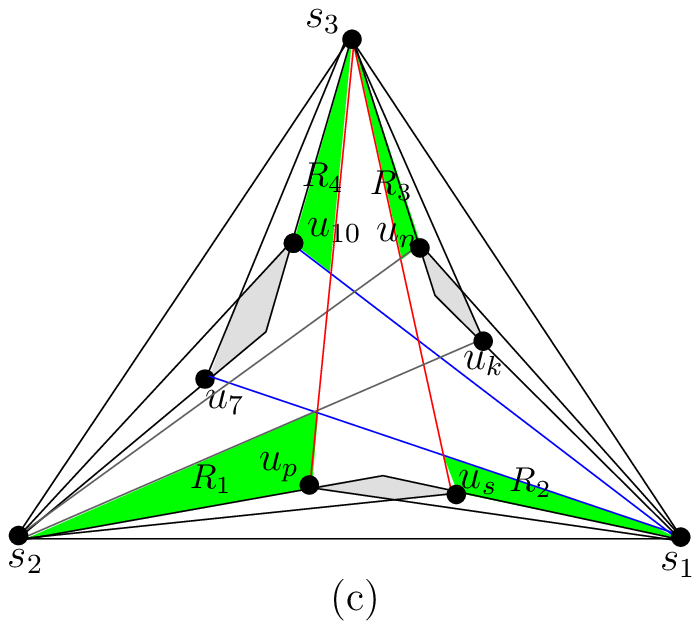}\\
\end{minipage}
\caption{(a) Illustration for the proof of Lemma \ref{lm:lmsevenempty}, (b) Diamond arrangement $D\{u_7, u_{10}\}$, (c) Arrangement of diamonds $D\{u_7,u_{10}\}$, $D\{u_k,u_n\}$, and $D\{u_p,u_s\}$ in $\bbI( s_1s_2s_3)$.}
  \label{fig:thull}\vspace{-0.15in}
\end{figure*}

The following lemma proves the admissibility of $S$ in the remaining cases.

\begin{lemma}If for all $i\ne j\ne k\in\{1, 2, 3\}$, $Cone(s_ku_{\alpha}s_j)\cap S$ and $Cone(s_ku_{\beta}s_j)\cap S$ are empty, where
$u_{\alpha}$, $u_{\beta}$ are the seventh and eighth angular neighbors of
$\overrightarrow{s_is_j}$ in $Cone(s_js_is_k)$, respectively, then $S$ is
admissible. \label{lm:lmseveneight}
\end{lemma}

\begin{proof}Lemmas \ref{lm:lmeightempty} and \ref{lm:lmsevenempty} imply that $S$
is admissible unless the interiors of $ s_2u_7s_{3}$, $ s_2u_8s_{3}$, $
s_2u_9s_{3}$, and $ s_2u_{10}s_{3}$ are empty in $S$. Thus, points $u_7, u_{8},
u_{9}, u_{10}$ must be arranged inside $CH(S)$ as shown in Figure
\ref{fig:thull}(b). We call such a set of 4 points a {\it diamond} and denote
it by  $D\{u_7, u_{10}\}$. Note that, $|\bbI( s_1s_2u_7)\cap S|=|\bbI(
s_1s_3u_{10})\cap S|=6$.

Since $Cone(s_1u_7p_{27})\cap S$ is non-empty by Assumption \ref{assumption4},
$u_7$ cannot be the seventh, eighth, ninth, and tenth angular neighbors of $\overrightarrow{s_2s_1}$ in
$Cone(s_1s_2s_3)$. Let $u_k$ be the seventh angular neighbor of
$\overrightarrow{s_2s_1}$ in $Cone(s_1s_2s_3)$. Suppose that $u_k\in \bbI(u_7s_2s_1)$. Then
we have $|\bbI(s_1u_kp_{2k})\cap S|\geq 1$, as $|\bbI(u_7s_1s_2)\cap S|=6$. Hence,
$|\bbV(CH(\mathcal H_c(s_2u_k, s_1)\cap S))|\geq 4$, and since $|\mathcal H_c(s_2u_k, s_1)\cap S|=9$,
the admissibility of $S$, in this case, follows from Corollary \ref{corollary:nine_points}.

Therefore, it can be assumed that the seventh angular neighbor of $\overrightarrow{s_2s_1}$, that is,
$u_k$ lies in $\bbI(p_{27}u_7s_1)\cap S$. Then Lemmas \ref{lm:lmeightempty} and \ref{lm:lmsevenempty}
imply that the eighth, ninth, and tenth angular neighbors of $\overrightarrow{s_2s_1}$ are in
$Cone(s_1u_7p_{27})$. Let $u_l$, $u_m$, and $u_n$ denote the eighth, ninth, and
tenth angular neighbors of $\overrightarrow{s_2s_1}$ in $Cone(s_1s_2s_{3})$,
respectively. From similar arguments as before, these three points along with
the point $u_k$ form a diamond, $D\{u_k, u_n\}$, which is disjoint from
diamond $D\{u_7, u_{10}\}$ (see Figure \ref{fig:thull}(c)).

Let  $u_s$ be the seventh angular neighbor of  $\overrightarrow{s_3s_1}$ in
$Cone(s_1s_3u_{10})$ as shown in Figure \ref{fig:thull}(c). Again, Assumption
\ref{assumption4} and the same logic as before implies $S$ is admissible if
$u_{10}$ is the eighth, ninth or tenth angular neighbor of
$\overrightarrow{s_3s_1}$ in $Cone(s_1s_3u_{10})$. Let $u_r$, $u_q$, and $u_p$
be the eighth, ninth, and tenth angular neighbors of $\overrightarrow{s_3s_1}$
in $Cone(s_1s_3u_{10})$, respectively. As before, these three points along with
the point $u_s$, form another diamond $D\{u_p, u_s\}$, which disjoint from both
$D\{u_7, u_{10}\}$ and $D\{u_k, u_n\}$.

Let $R_1, R_2, R_3, R_4$ be the shaded regions inside $CH(S)$, as shown in Figure \ref{fig:thull}(c).
To begin with suppose that $|R_1\cap S|\geq 1$. Let $u_z$ be the first angular neighbor of $\overrightarrow{u_ps_3}$ in $Cone(p_{2p}u_ps_3)$. Note that $|\mathcal H_c(u_pu_z, s_3)\cap S|=10$ and $\bbI(s_2u_zu_p)\cap S$ is non-empty, as $|R_1\cap S|\geq 1$. This implies that $u_p$ is 5-redundant in $\mathcal H_c(u_pu_z, s_3)\cap S$. Therefore, the 5-hole contained in $(\mathcal H(u_pu_z, s_3)\cap S)\cup\{u_z\}$ is disjoint from the 5-hole contained in $(\mathcal H(u_pu_z, s_1)\cap S)\cap\{u_p\}$.  Therefore, assume that $|R_1\cap S|=0$. This implies that $|R_4\cap S|=2$, as $|\bbI(s_2s_3u_p)\cap S|=6$. The admissibility of $S$ now follows from exactly similar arguments by taking the nearest angular neighbor of $\overrightarrow{u_{10}s_1}$ in $Cone(s_1u_{10}p_{310})$. \hfill $\Box$
\end{proof}

Since all the different cases have been considered, the proof of the case
$|\bbV(CH(S))|=3$, and hence the theorem is finally completed.

\section{Proof of Theorem \ref{th:twomplusnine}}
\label{proof2m+9}

Let $S$ be any set of $2m+9$ points in the plane in general position, and
$u_1$, $u_2$, and $w_m$ be vertices of $CH(S)$ such that $u_1u_2$ and
$u_1w_m$ are edges of $CH(S)$. We label the points in the set $S$
inductively as follows.

\begin{description}
\item[{\it (i)}]{$u_i$ be the $(i-2)$-th angular neighbor of $\overrightarrow{u_1u_2}$ in $Cone(w_mu_1u_2)$, where $i\in\{3, 4, \ldots, m\}$.}

\item[{\it (ii)}]{$v_i$ be the $i$-th angular neighbor of $\overrightarrow{u_1u_m}$ in $Cone(w_mu_1u_m)$, where $i\in\{1, 2, \ldots, 9\}$.}

\item[{\it (iii)}]{$w_i$ be the $i$-th angular neighbor of $\overrightarrow{u_1v_9}$ in $Cone(w_mu_1v_9)$, where $i\in\{1, 2, \ldots, m\}$.}
\end{description}

Therefore, $S=U\cup V \cup W$, where $U=\{u_1,u_2, \dots,u_m\}$, $ V=\{v_1,v_2,
\dots,v_9\}$, and $W=\{w_1,w_2, \dots,w_m\}$.

A disjoint convex partition of $S$ into three subsets $S_1, S_2, S_3$ is said
to be a {\it separable} partition of $S$ (or {\it separable} for $S$) if
$|S_1|=|S_3|=m$ and the set of 9 points $S_2$ contains a 5-hole. The set $S$ is
said to be {\it separable} if there exists a partition which is separable for
$S$. For proving Theorem \ref{th:twomplusnine} we have to identify a separable
partition for every set of $2m+9$ points in the plane in general position. It
is clear, from Corollary \ref{corollary:nine_points}, that $S$ is separable
whenever $|\bbV(CH(V))|\geq 4$.

Let $T=V\backslash\{v_9\}\cup\{u_1\}$. If $|\bbV(CH(T))|\geq 6$, $u_1$ is
$5$-redundant in $T$ and $S_1=U,~S_2=V$, and $S_3=W$ is a separable partition
of $S$.

Therefore, assume that $|\bbV(CH(T))|\leq 5$. The three cases based on the size of
$|\bbV(CH(T))|$ are considered separately in the following lemmas.

\begin{figure*}[h]
\centering
\begin{minipage}[c]{0.5\textwidth}
\centering
\includegraphics[width=2.5in]
    {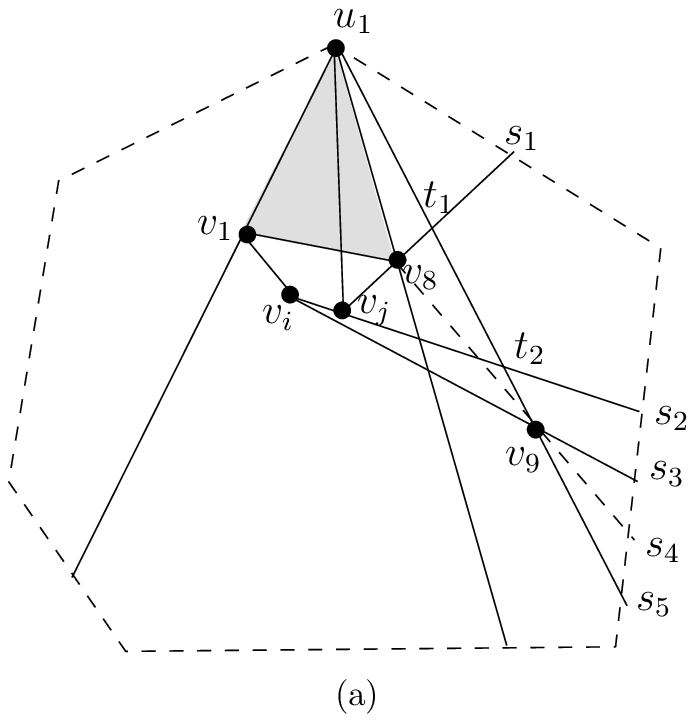}\\
\end{minipage}%
\begin{minipage}[c]{0.5\textwidth}
\centering
\includegraphics[width=2.4in]
    {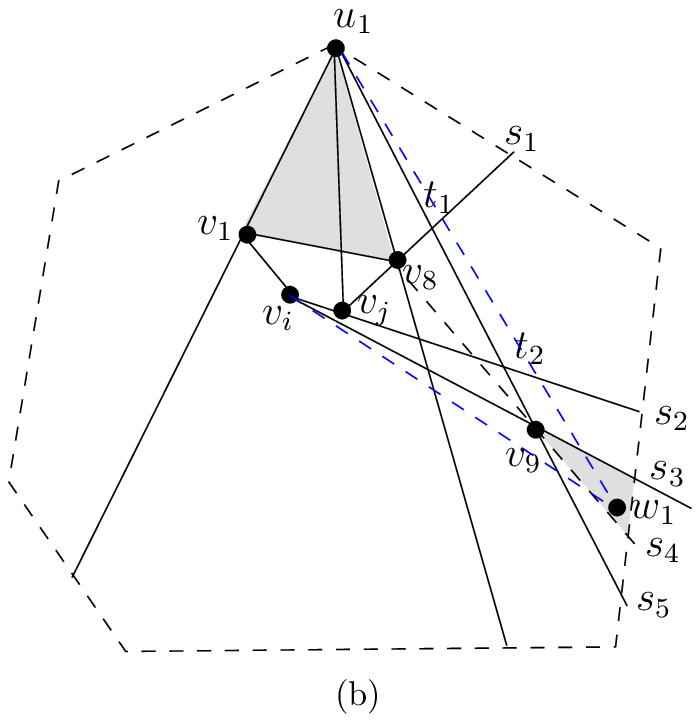}\\
\end{minipage}
\caption{Illustrations for the proof of Lemma \ref{lm:twomplusnine_ch5}.}
  \label{fig:case1} \vspace{-0.15in}
\end{figure*}

\begin{lemma}$S$ is separable whenever $|\bbV(CH(T))|=5$.
\label{lm:twomplusnine_ch5}
\end{lemma}

\begin{proof}
Let $\{u_1, v_1, v_i, v_j, v_8\}$ be the vertices of the convex hull of $T$. It
suffices to assume that $\bbI( u_1v_1v_i)$ and $\bbI( u_1v_1v_8)$ are empty in
$S$, otherwise either $v_1$ or $u_1$ is, respectively, 5-redundant and $S$ is
separable. Let the lines $\overrightarrow{v_jv_8}$ and  $\overrightarrow
{v_iv_j}$ intersect $\overrightarrow {u_1v_9}$ at the points $t_1$, $t_2$, and
$CH(S)$ at the points $s_1$, $s_2$, respectively (Figure \ref{fig:case1}(a)).
Now, we consider the following cases based on the location of the point $v_9$
on the line segment $u_1s_5$, where $s_5$ is the point where
$\overrightarrow{u_1v_9}$ intersects the boundary of $CH(S)$.

\begin{description}
\item[{\it Case} 1:]$v_9$ lies on the line segment $u_1t_2$.
This implies, $|\bbV(CH(V))|\geq 4$ and by Corollary \ref{corollary:nine_points},
$S_1=U$, $S_2=V$, and $S_3=W$ is a separable partition of $S$.

\item[{\it Case} 2:]$v_9$ lies on the line segment $t_2s_5$. Let $s_3$ and $s_4$ be the points where
the lines $\overrightarrow{v_iv_9}$ and $\overrightarrow{v_8v_9}$ intersects
$CH(S)$, respectively. (Note that if $v_9=s_5$, then the points $s_3$ and $s_4$ coincide with the point $v_9$.)
If  $Cone(u_1t_1s_1)\cap S$ is non-empty, let $w_q$ be
the first angular neighbor of $\overrightarrow{v_8u_1}$ in $Cone(u_1t_1s_1)$.
This implies, $|\bbV(CH(V\backslash \{v_1, v_9\}\cup\{u_1, w_q\}))|\geq 5$ and  by
Corollary \ref{corollary:nine_points} $S_1=U\backslash\{u_1\}\cup\{v_1\}$,
$S_2=V\backslash \{v_1, v_9\}\cup\{u_1, w_q\}$, and
$S_3=W\backslash\{w_q\}\cup\{v_9\}$ is a separable partition of $S$. So, assume
that $Cone(u_1t_1s_1)\cap S$ empty.

\begin{description}

\item[{\it Case} 2.1:] $Cone(s_1v_js_2)\cap W$ is non-empty.
Let $w_q$ be the first angular neighbor of $\overrightarrow{v_js_1}$ in
$Cone(s_1v_js_2)$. Then, $|\bbV(CH(V\backslash\{v_9\}\cup\{w_q\}))|\geq4$, and the
partition, $S_1=U$, $S_2=V\backslash \{v_9\}\cup\{w_q\}$, and
$S_3=W\backslash\{w_q\}\cup\{v_9\}$ is separable for $S$.

\item[{\it Case} 2.2:] $Cone(s_1v_js_2)\cap W$ is empty and $Cone(s_5v_9s_4)\cap W$ is non-empty. Let $w_q$ be the first angular neighbor of $\overrightarrow{v_9s_5}$ in $Cone(s_5v_9s_4)$. Observe that $|\bbV(CH(V\cup \{w_q\}))|\geq 4$ and $\bbI(v_8v_9w_q)\cap S$ is empty. Now, if $|\bbV(CH(V\cup \{w_q\}))|\geq 5$, then $v_1$ is clearly 5-redundant in $V\cup\{w_q\}$. Otherwise, Corollary \ref{cor:5redun} now implies that $v_1$ is 5-redundant in $V\cup\{w_q\}$. Therefore, the partition $S_1=U\backslash\{u_1\}\cup \{v_1\}$, $S_2=V\backslash\{v_1\}\cup\{w_q\}$, and $S_3=W\backslash\{w_q\}\cup\{u_1\}$ is separable for $S$.

\item[{\it Case} 2.3:] $Cone(s_1v_js_2)\cap W$ and $Cone(s_5v_9s_4)\cap W$ are both empty.
If $w_1$, the nearest angular neighbor of $\overrightarrow{u_1s_5}$ in $W$,
lies in $Cone(s_2v_is_3)$, $|\bbV(CH(V\backslash\{v_1\}\cup\{u_1, w_1\}))|=4$ and
$u_1$ is 5-redundant in $V\backslash\{v_1\}\cup\{u_1, w_1\}$ by Corollary
\ref{cor:5redun}. Therefore, $S_1=U\backslash\{u_1\}\cup\{v_1\}$,
$S_2=V\backslash\{v_1\}\cup\{w_1\}$, and $S_3 =W\backslash\{w_1\}\cup\{u_1\}$
is separable for $S$. Finally, consider that $w_1\in Cone(s_4v_9s_3)$ and let
$Z=V\backslash\{v_1\}\cup\{u_1, w_1\}$. Observe that $|\bbV(CH(Z))|=3$ (Figure
\ref{fig:case1}(b)). Now, since $|Z|=10$, $Z$ must contain a 5-hole. Note that
since $\bbI(u_1v_1v_8)$ is assumed to be empty in $S$, it follows that all the four vertices of the 4-hole
$u_1v_8v_9w_1$ cannot be a part of any 5-hole
in $Z$. Moreover, there cannot be a 5-hole in $Z$ with the points $u_1, v_9,
w_1$ or the points $u_1, v_8, v_9$ as vertices, since $Cone(s_5u_1w_1)$ and $Cone(u_1w_1v_8)$ are empty in $Z$. Emptiness of $Cone(s_5u_1w_1)\cap Z$ and $Cone(u_1w_1v_8)\cap Z$ also implies that there cannot be a 5-hole in $Z$ with both the points $u_1$ and $w_1$ as vertices. Thus, either $u_1$ or $w_1$ is
5-redundant in $Z$, and separability of $S$ follows. \hfill $\Box$
\end{description}
\end{description}
\end{proof}

\begin{figure*}[h]
\centering
\begin{minipage}[c]{0.5\textwidth}
\centering
\includegraphics[width=2.75in]
    {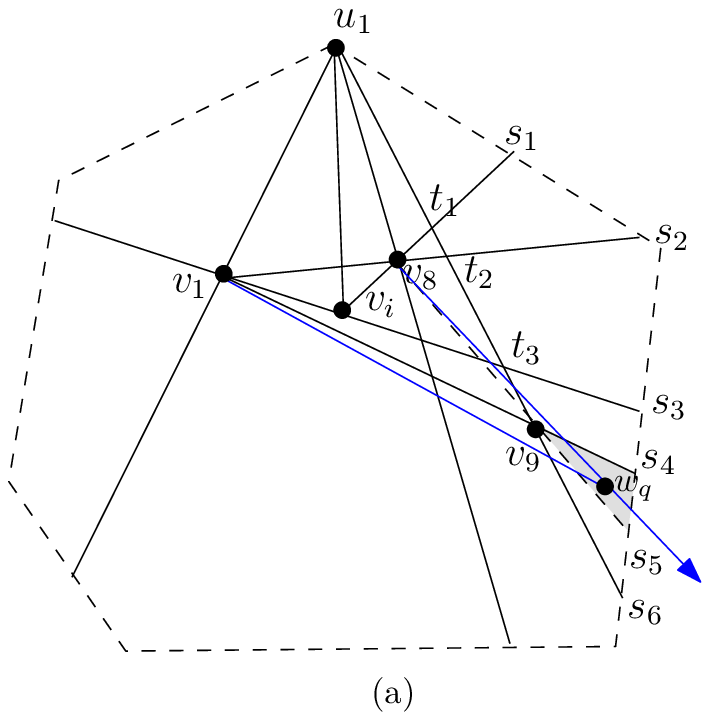}\\
\end{minipage}%
\begin{minipage}[c]{0.5\textwidth}
\centering
\includegraphics[width=2.58in]
    {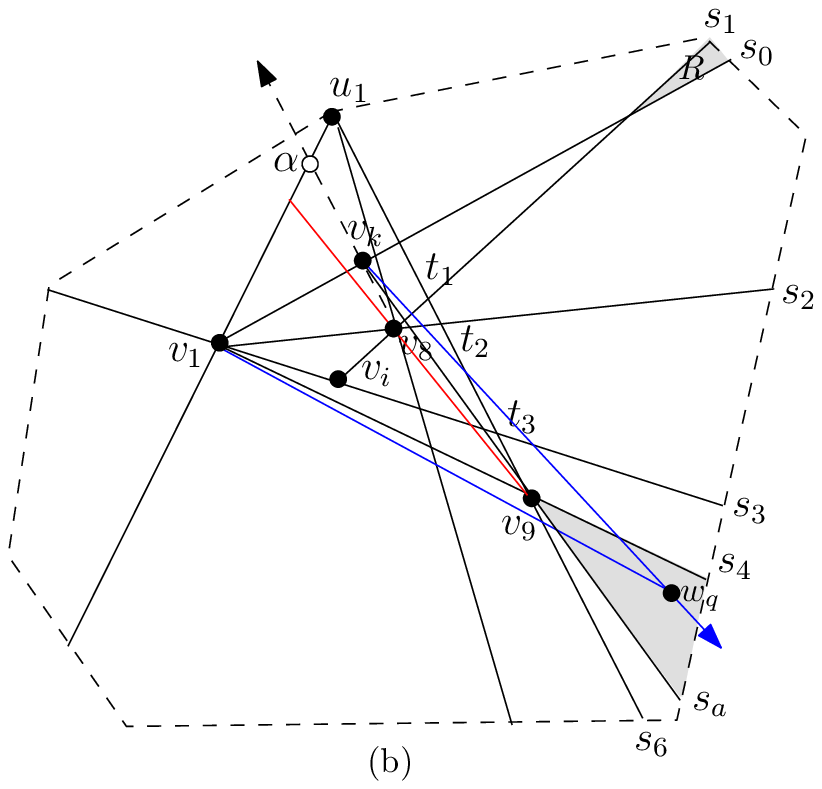}\\
\end{minipage}
\caption{Illustrations for the proof of Lemma \ref{lm:twomplusnine_ch4}: {\it
Case} 1 and {\it Case} 2.} \label{fig:case2} \vspace{-0.15in}
\end{figure*}

\begin{lemma}$S$ is separable whenever $|\bbV(CH(T))|=4$.
\label{lm:twomplusnine_ch4}
\end{lemma}

\begin{proof}Suppose $\{u_1, v_1, v_i, v_8\}$ are the
vertices of the convex hull of $T$. Let the lines $\overrightarrow{v_iv_8}$,
$\overrightarrow{v_1v_8}$, and $\overrightarrow{v_1v_i}$ intersect
$\overrightarrow{u_1v_9}$ at the points $t_1$, $t_2$, $t_3$, and $CH(S)$ at the
points $s_1$, $s_2$, $s_3$, respectively (see Figure \ref{fig:case2}(a)). If
$v_9$ lies on the line segment $u_1t_1$ or $t_2t_3$, then $|\bbV(CH(V))|\geq 4$ and
$S_1=U$, $S_2=V$, and $S_3=W$ is separable for $S$. So, assume that $v_9$ lies
on the line segment $t_1t_2$, or on the line segment $t_3s_6$, where $s_6$ is
the point of intersection of $\overrightarrow{u_1v_9}$ and $CH(S)$. Now, we
consider the following cases.

\begin{description}
\item[{\it Case} 1:] $v_9$ lies on the line segment $t_3s_6$, and $\bbI( u_1v_1v_8)\cap S$ is empty.
Let $s_4$ and $s_5$ be the points where $\overrightarrow{v_1v_9}$ and
$\overrightarrow{v_8v_9}$ intersect the boundary of $CH(S)$, respectively.

\begin{description}
\item[{\it Case} 1.1:] $Cone(u_1v_8s_1)\cap W$ is non-empty.
If $w_q$ be the first angular neighbor of $\overrightarrow{v_8u_1}$ in
$Cone(u_1v_8s_1)$, then $|\bbV(CH(V\backslash\{v_9\}\cup\{w_q\}))|=4$. Hence,
$S_1=U$, $S_2= V\backslash\{v_9\}\cup\{w_q\}$, and
$S_3=W\backslash\{w_q\}\cup\{v_9\}$ is a separable partition.

\item[{\it Case} 1.2:] $Cone(u_1v_8s_1)\cap W$ is empty, and $Cone(s_6v_9s_5)\cap W$ is non-empty.
Let $w_q$ be the first angular neighbor of $\overrightarrow{v_9s_6}$ in
$Cone(s_6v_9s_5)$. Note that $CH(V\cup\{w_q\})$ is a quadrilateral and
$\bbI(v_8v_9w_q)\cap S$ is empty. This implies that $v_1$ is 5-redundant in
$V\cup\{w_q\}$ by Corollary \ref{cor:5redun}. Therefore,
$S_1=U\backslash\{u_1\}\cup\{v_1\}$, $S_2=V\backslash\{v_1\}\cup\{w_q\}$, and
$S_3=W\backslash\{w_q\} \cup\{u_1\}$ is separable for $S$.

\item[{\it Case} 1.3:] Both $Cone(u_1v_8s_1)\cap W$ and $Cone(s_6v_9s_5)\cap W$ are empty, but $Cone(s_5v_8s_2)\cap W$ is non-empty.
Let $w_q$ be the first angular neighbor of $\overrightarrow{v_8v_9}$ in $Cone(s_5v_8s_2)$.
To begin with, assume $w_q\in Cone(s_5v_8s_2)\backslash Cone(s_5v_9s_4)$. Then
$|\bbV(CH(V\cup\{w_q\}))|\geq 4$ and $V\cup\{w_q\}$ contains a 5-hole. Now, by
Corollary \ref{cor:5redun}, either $v_1$ or $w_q$ is 5-redundant in
$V\cup\{w_q\}$, and the separability of $S$ is immediate. Otherwise, $w_q\in
Cone(s_5v_9s_4)$, and $|\bbV(CH(V\cup\{w_q\}))|=3$ (Figure \ref{fig:case2}(a)). Now,
$V\cup\{w_q\}$ contains a 5-hole and at least one of $v_1$, $v_8$, and $w_q$ is
5-redundant in $V\cup\{w_q\}$. If $w_q$ is 5-redundant, the separability of $S$
is immediate. If $v_1$ is 5-redundant, the partition
$S_1=U\backslash\{u_1\}\cup\{v_1\}$, $S_2=V\backslash\{v_1\}\cup\{w_q\}$, and
$S_3=W\backslash\{w_q\}\cup\{u_1\}$ is a separable partition of $S$. Finally,
if $v_8$ is 5-redundant, then the partition $S_1=U$,
$S_2=V\backslash\{v_8\}\cup\{w_q\}$, and $S_3=W\backslash\{w_q\}\cup\{v_8\}$ is
a separable partition of $S$.

\item[{\it Case} 1.4:] $W\subset Cone(s_1v_8s_2)$. Let $w_q$ be the nearest angular neighbor of $\overrightarrow{v_is_1}$ in $Cone(s_1v_is_3)$. If $\bbI( u_1v_1v_i)\cap S$ is non-empty, then $|\bbV(CH(V\backslash\{v_1, v_9\}\cup\{u_1, w_q\}))|\geq 4$ and the partition $S_1=U\backslash\{u_1\}\cup\{v_1\}$, $S_2=V\backslash\{v_1, v_9\}\cup\{u_1, w_q\}$, and    $S_3=W\backslash\{w_q\}\cup\{v_9\}$ is separable for $S$. Otherwise, assume $\bbI( u_1v_1v_i)\cap S$ is empty. Let $w_1$ be the first angular neighbor of $\overrightarrow{u_1s_6}$ in $W$. Then, $|\bbV(CH(V\backslash\{v_1\}\cup\{w_1\}))|\geq 4$, and the partition $S_1=U\backslash\{u_1\}\cup\{v_1\}$, $S_2=V\backslash\{v_1\}\cup\{w_1\}$, and $S_3=W\backslash\{w_1\}\cup\{u_1\}$ is separable for $S$.

\end{description}

\begin{figure*}[h]
\centering
\begin{minipage}[c]{0.5\textwidth}
\centering
\includegraphics[width=2.5in]
    {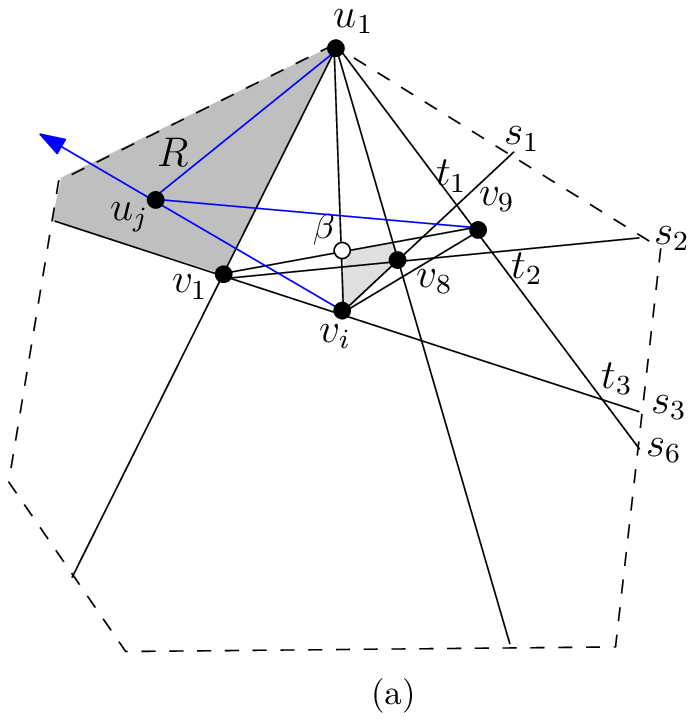}\\
\end{minipage}%
\begin{minipage}[c]{0.5\textwidth}
\centering
\includegraphics[width=2.5in]
    {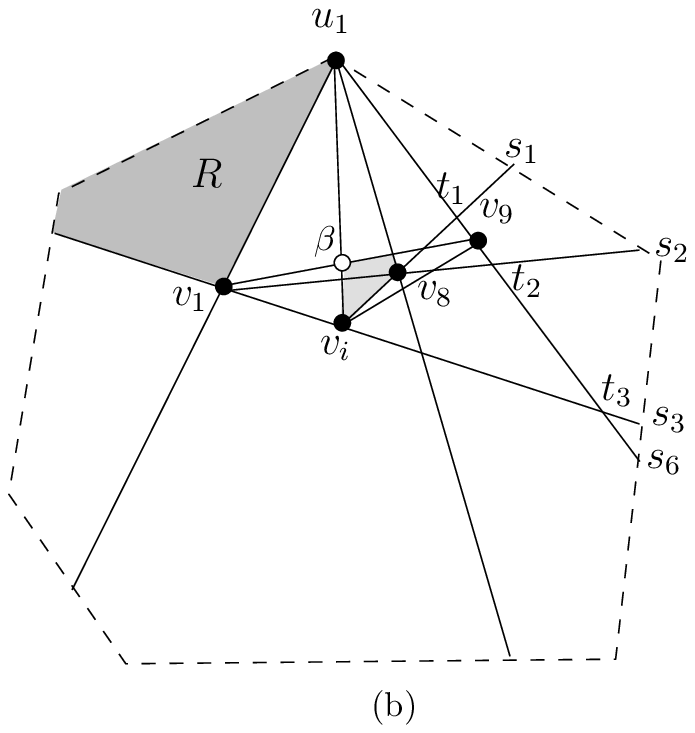}\\
\end{minipage}
\caption{Illustrations for the proof of Lemma \ref{lm:twomplusnine_ch4}: {\it
Case} 3.} \label{fig:case2_t1t2}\vspace{-0.15in}
\end{figure*}

\item[{\it Case} 2:]$v_9$ lies on the line segment $t_3s_6$, and $\bbI( u_1v_1v_8)\cap S$ is
non-empty. Let $v_k$ be the first angular neighbor of $\overrightarrow{v_8u_1}$
in $Cone(u_1v_8v_1)$, and let $s_0$, $s_4$ and $s_a$ be the points where
$\overrightarrow{v_1v_k}$, $\overrightarrow{v_1v_9}$ and
$\overrightarrow{v_kv_9}$ intersect $CH(S)$, respectively.
Note that if $v_k\in \overline\mathcal H(v_9v_8, u_1)\cap V$, then $|\bbV(CH(V))|\geq 4$ and the separability
of $S$ is immediate. Therefore, assume that $v_k\in \mathcal H(v_9v_8, u_1)\cap V$ (see Figure \ref{fig:case2}(b)). Let $\alpha$ be the point where $\overrightarrow{v_8v_k}$ intersects $\overrightarrow{u_1v_1}$. If $\bbI(v_1v_k\alpha)\cap V$ is non-empty, then $|CH(V)|\geq 5$, and the separability of $S$ is immediate. Therefore, assume that $\bbI(v_1v_k\alpha)\cap V$ is empty, that is, $\bbI(v_1v_ku_1)\cap V$ is empty.

\begin{description}
\item[{\it Case} 2.1:] $Cone(s_6v_9s_a)\cap W$ is non-empty.
Let $w_q$ be the first angular neighbor of $\overrightarrow{v_9s_6}$ in
$Cone(s_6v_9s_a)$. Then $|\bbV(CH(V\cup\{w_q\}))|=4$ and by Corollary
\ref{cor:5redun} either $v_1$ or $v_9$ is 5-redundant in $V\cup\{w_q\}$. The
separability of $S$ now follows easily.

\item[{\it Case} 2.2:] $Cone(s_6v_9s_a)\cap W$ is empty and $Cone(s_0v_ks_a)\cap W$ is non-empty.
Let $w_q$ be the first angular neighbor of $\overrightarrow{v_kv_9}$ in
$Cone(s_0v_ks_a)$. If $w_q\in Cone(s_0v_ks_a)\backslash Cone(s_4v_9s_a)$ then
$|\bbV(CH(V\backslash\{v_1\}\cup \{w_q\}))|\geq 4$, and the separability of $S$ is
immediate. Otherwise, assume $w_q\in Cone(s_4v_9s_a)$. Then
$|\bbV(CH(V\cup\{w_q\}))|=3$ and either $v_1, v_k$, or $w_q$ is 5-redundant in
$V\cup\{w_q\}$, and the separability of $S$ is immediate.

\item[{\it Case} 2.3:] Both $Cone(s_6v_9s_a)\cap W$ and $Cone(s_0v_ks_a)\cap W$ are
empty, but $Cone(u_1v_ks_0)\cap W$ is non-empty. Now, if $Cone(u_1v_8s_1)\cap
W$ is non-empty, the partition $S_1=U\backslash\{u_1\}\cup\{v_1\}$, $S_2=V\backslash \{v_1, v_9\}\cup\{u_1, w_i\}$,
and $S_3=W\backslash\{w_i\}\cup\{v_9\}$ is separable for $S$, where $w_i$ is the first angular neighbor of $\overrightarrow{v_8u_1}$ in $Cone(u_1v_8s_1)\cap W$. Therefore, assume that
$Cone(u_1v_8s_1)\cap W$ is empty. This implies, $W\subset R\cap S$, where $R$
is the shaded region as shown in Figure \ref{fig:case2}(b). Let $w_q$ be the
nearest angular neighbor of $\overrightarrow{v_iv_8}$ in $Cone(s_1v_is_3)$. If
$\bbI( u_1v_1v_i)\cap S$ is non-empty, then $|\bbV(CH(V\backslash\{v_1,
v_9\}\cup\{u_1, w_q\}))|\geq 4$ and the partition
$S_1=U\backslash\{u_1\}\cup\{v_1\}$, $S_2=V\backslash\{v_1, v_9\}\cup\{u_1,
w_q\}$, and $S_3=W\backslash\{w_q\}\cup\{v_9\}$ is separable for $S$.
Otherwise, assume $\bbI( u_1v_1v_i)\cap S$ is empty. Let $w_1$ be the first
angular neighbor of $\overrightarrow{u_1s_6}$ in $W$. Then,
$|\bbV(CH(V\backslash\{v_1\}\cup\{w_1\}))|\geq 4$, and the partition
$S_1=U\backslash\{u_1\}\cup\{v_1\}$, $S_2=V\backslash\{v_1\}\cup\{w_1\}$, and
$S_3=W\backslash\{w_1\}\cup\{u_1\}$ is separable for $S$.

\end{description}

\item[{\it Case} 3:]$v_9$ lies on the line segment $t_1t_2$. Observe that if either $u_1$ or $v_1$ is 5-redundant in $V\cup\{u_1\}$, then the separability of $S$ is immediate. Therefore, from Corollary \ref{cor:5redun}, it suffices to assume that all the points inside $CH(V\cup\{u_1\})$ must lie in $\bbI(v_9v_i\beta)$, where $\beta$ is the point of intersection of the diagonals of the quadrilateral $u_1v_1v_iv_9$. Next, suppose that $R\cap S$ is non-empty, where $R$ is the shaded region inside $CH(S)$ as shown in Figure \ref{fig:case2_t1t2}(a). Let $u_j\in R\cap S$ be the first angular neighbor of $\overrightarrow{v_iu_1}$ in $R$. Then $|\bbV(CH(V\backslash\{v_1\} \cup\{u_1, u_j\}))|=4$ and $v_i$ is 5-redundant in $V\backslash\{v_1\} \cup\{u_1, u_j\}$, since $\bbI(u_jv_iv_9)\cap S$ is non-empty (Corollary \ref{cor:5redun}). Hence, the partition of $S$ given by $S_1=U\backslash\{u_1, u_j\} \cup\{v_1, v_i\}, S_2=V\backslash\{v_1, v_i\} \cup\{u_1, u_j\}, S_3=W$ is separable. On the other hand, if $R\cap S$ is empty, then the partition $S_1=U\backslash\{u_1\}\cup\{v_i\}$, $S_2=V\backslash\{v_i\}\cup\{u_1\}$, and $S_3=W$ is separable, since $v_i$ is 5-redundant in $V\cup\{u_1\}$ by Corollary \ref{cor:5redun} (see Figure \ref{fig:case2_t1t2}(b)).
\end{description}
\end{proof}

\begin{lemma}$S$ is separable whenever $|\bbV(CH(T))|=3$.
\label{lm:twomplusnine_ch3}
\end{lemma}

\begin{proof} Let $\bbV(CH(T))=\{u_1, v_1, v_8\}$.
Let $v_i$ and $v_j$ be the first angular neighbors of $\overrightarrow{v_8u_1}$
and $\overrightarrow{v_8v_1}$ respectively in $Cone(u_1v_8v_1)$. Let
$\overrightarrow{v_jv_8}$ and $\overrightarrow{v_iv_8}$ intersect
$\overrightarrow{u_1v_9}$ at $t_1$ and $t_2$, respectively (Figure
\ref{fig:case3}(a)). If $v_9$ lies on the line segment $u_1t_1$,
$|\bbV(CH(V\backslash\{v_1\}\cup \{u_1\}))|\geq 4$ and by Corollary
\ref{corollary:nine_points}, $V\backslash\{v_1\}\cup\{u_1\}$ contains a 5-hole.
Thus, $S_1=U\backslash\{u_1\}\cup\{v_1\}$, $S_2=V\backslash\{v_1\}\cup\{u_1\}$,
and $S_3=W$ is a separable partition of $S$. Similarly, if $v_9$ lies on the
line segment $t_2s_4$, where $s_4$ is the point where $\overrightarrow{u_1v_9}$
intersects the boundary of $CH(S)$, then $|\bbV(CH(V))|\geq 4$, and $S_1=U$, $S_2=V$,
and $S_3=W$ is separable for $S$.

\begin{figure*}[h]
\centering
\begin{minipage}[c]{0.5\textwidth}
\centering
\includegraphics[width=2.5in]
    {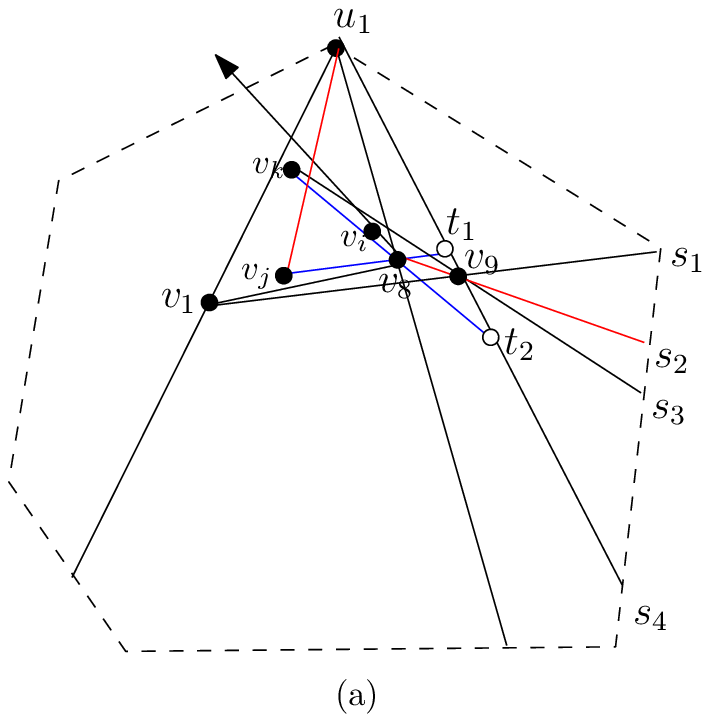}\\
\end{minipage}%
\begin{minipage}[c]{0.5\textwidth}
\centering
\includegraphics[width=2.5in]
    {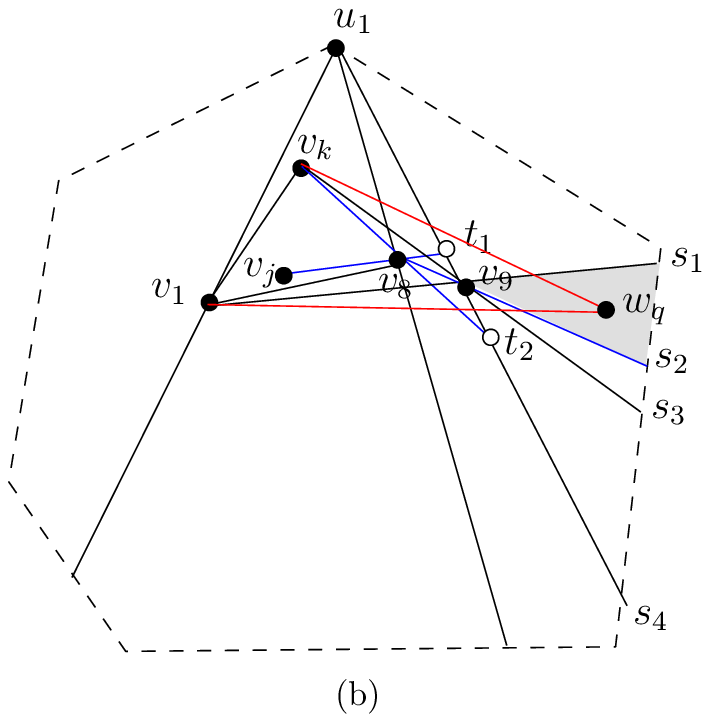}\\
\end{minipage}
\caption{Illustrations for the proof of Lemma \ref{lm:twomplusnine_ch3}.}
  \label{fig:case3} \vspace{-0.15in}
\end{figure*}

Therefore, $v_9$ lies on the line segment $t_1t_2$. Clearly, $S$ is separable
unless $|\bbV(CH(V))|=3$. Let $\bbV(CH(V))=\{v_1, v_k, v_9\}$. (Note that $v_k$ need
not be the point $v_i$ as shown in Figure \ref{fig:case3}(a)). Let $s_1$, $s_2$, and $s_3$ be the points where $\overrightarrow{v_1v_9}$, $\overrightarrow{v_8v_9}$, and $\overrightarrow{v_kv_9}$ intersect $CH(S)$,
respectively. Now, we have the following cases:

\begin{description}
\item[{\it Case} 1:] $Cone(u_1v_8t_1)\cap S$ is non-empty. Let $w_q$ be the first angular neighbor of
$\overrightarrow{v_8u_1}$, in $Cone(u_1v_8t_1)$. This implies,
$|\bbV(CH(V\backslash\{v_1, v_9\}\cup\{u_1, w_q\}))|\geq 4$, and
$S_1=U\backslash\{u_1\}\cup\{v_1\}$, $S_2=V\backslash\{v_1, v_9\}\cup\{u_1,
w_q\}$, and $S_3=W\backslash\{w_q\}\cup\{v_9\}$ is a separable partition of
$S$.

\item[{\it Case} 2:] $Cone(u_1v_8t_1)\cap S$ is empty and $Cone(s_4v_9s_3)\cap S$ is non-empty.
Suppose, $w_q$ is the first angular neighbor of $\overrightarrow{v_9s_4}$ in
$Cone(s_4v_9s_3)$. Since $|\bbV(CH(V\cup\{w_q\}))|\geq 4$, either $v_1$ or $v_9$ is
5-redundant in $V\cup\{w_q\}$ by Corollary \ref{cor:5redun}. Thus, either
$S_1=U\backslash\{u_1\}\cup\{v_1\}$, $S_2=V\backslash\{v_1\}\cup\{w_q\}$, and
$S_3=W\backslash\{w_q\}\cup\{u_1\}$ or $S_1=U$,
$S_2=V\backslash\{v_9\}\cup\{w_q\}$, and $S_3=W\backslash\{w_q\}\cup\{v_9\}$
is, respectively, separable for $S$.

\item[{\it Case} 3:] $Cone(u_1v_8t_1)\cap S$ and $Cone(s_4v_9s_3)\cap S$ are empty but  $Cone(s_3v_9s_2)\cap S$ is non-empty.
If $w_q$ is the first angular neighbor of $\overrightarrow{v_9s_3}$ in
$Cone(s_3v_9s_2)$, then
$v_1v_jv_8v_9w_q$ is a 5-hole, and $S_1=U$,
$S_2=V\backslash\{v_k\}\cup\{w_q\}$, and $S_3=W\backslash\{w_q\}\cup\{v_k\}$ is
separable for $S$.

\item[{\it Case} 4:] The three sets $Cone(u_1v_8t_1)\cap S$, $Cone(s_4v_9s_3)\cap S$, and $Cone(s_3v_9s_2)\cap S$ are all empty,
but $Cone(t_1v_8s_2)\cap S$ is non-empty. Let $w_q$ be the first angular
neighbor of $\overrightarrow{v_kv_9}$ in $Cone(u_1v_kv_9)$. Clearly, $w_q\in
Cone(t_1v_8s_2)$.
\begin{description}
\item[{\it Case} 4.1:] $w_q\in Cone(t_1v_8s_2)\backslash Cone(s_2v_9s_1)$.
In this case, $|\bbV(CH(V\cup\{w_q\}))|=4$ and $v_1$ is 5-redundant in
$V\cup\{w_q\}$ by Corollary \ref{cor:5redun}. Then the partition
$S_1=U\backslash\{u_1\}\cup\{v_1\}$, $S_2=V\backslash\{v_1\}\cup\{w_q\}$, and
$S_3=W\backslash\{w_q\}\cup\{u_1\}$ is separable for $S$.

\item[{\it Case} 4.2:] $w_q\in Cone(s_2v_9s_1)$ (see Figure \ref{fig:case3}(b)). Let $Z=V\cup\{w_q\}$.
Observe, $|\bbV(CH(Z))|=3$ and $Z$ must contain a 5-hole, since $|Z|=10$. Now, either
$v_1$, $v_k$, or $w_q$ is 5-redundant in $Z$. If $w_q$ is 5-redundant, the
separability of $S$ is immediate. If $v_1$ is 5-redundant, the partition
$S_1=U\backslash\{u_1\}\cup\{v_1\}$, $S_2=V\backslash\{v_1\}\cup\{w_q\}$, and
$S_3=W\backslash\{w_q\}\cup\{u_1\}$ is a separable partition of $S$. Finally,
if $v_k$ is 5-redundant, then the partition $S_1=U$,
$S_2=V\backslash\{v_k\}\cup\{w_q\}$, and $S_3=W\backslash\{w_q\}\cup\{v_k\}$ is
a separable partition of $S$. \hfill $\Box$
\end{description}
\end{description}
\end{proof}

This finishes the analysis of all the different cases, and completes the proof of Theorem
\ref{th:twomplusnine}.


\section{Conclusion}
\label{c:conclusion}

In this paper we address problems concerning the existence of disjoint 5-holes
in planar point sets. We prove that every set of 19 points in the plane, in
general position, contains two disjoint 5-holes. Next, we show that any set of
$2m+9$ points in the plane can be subdivided into three disjoint convex regions
such that one contains a set of 9 points which contains a 5-hole, and the
others contain $m$ points each, where $m$ is a positive integer. Combining
these two results we show that the number of disjoint empty convex pentagons in
any set of $n$ points in the plane in general position, is at least
$\lfloor\frac{5n}{47}\rfloor$. This bound has been further improved to
$\frac{3n-1}{28}$ for infinitely many $n$.

In other words, we have shown that $H(5, 5)\leq 19$. This improves upon the
results of Hosono and Urabe \cite{hosono,kyotocggt}, where they showed $17\leq
H(5,5)\leq 20$. There is still a gap between the upper and lower bounds of
$H(5, 5)$, which probably requires a more complicated and detailed argument to
be settled.

However, we are still quite far from establishing non-trivial bounds on
$F_6(n)$ and $H(6, \ell)$, for $0\leq \ell\leq 6$, since the exact value of
$H(6)=H(6, 0)$ is still unknown. The best known bounds are $H(6)\leq ES(9)\leq
1717$ and $H(6)\geq 30$ by Gerken \cite{gerken} and Overmars
\cite{overmarsdcg}, respectively.


\bibliographystyle{IEEEbib}
\bibliography{strings,refs,manuals}

\end{document}